\documentclass[a4paper,reqno,10pt]{amsart}

\usepackage{amsmath}
\usepackage{amsthm}
\usepackage{amssymb}
\usepackage{amscd} %diagrams

\usepackage{bbm} %bold numbers
\usepackage{mathrsfs} %nice calligraphy

\usepackage{lmodern} 
\usepackage[english]{babel} %if writing in French
\usepackage[utf8]{inputenc} %uncomment these two lines
\usepackage[T1]{fontenc} %allows for hyphenation of words with accents
\usepackage{newtxtext,newtxmath}

\usepackage{manfnt} % The Bourbaki Z = \dbend

\usepackage{graphicx} %uncomment for inclusion of graphic images

\usepackage{verbatim} %DO NOT comment this line (comment environment)

\usepackage[Lenny]{fncychap} % other style = Rejne

\usepackage{sqrcaps} % square capital fonts

\usepackage{comment}

\usepackage[svgnames]{xcolor}
\usepackage{pdfcolmk}

\usepackage{tikz}
\usepackage{pifont}

\swapnumbers %theorems numbering appears in front of 'Theorem'

\newtheoremstyle{exercise} %for books or class notes
  {3pt} %space above
  {3pt} %space below
  {\small\rmfamily} %body font
  {
} %indent amount(empty=no indent,\parindent=para indent)
  {\rmfamily\scshape} %thm head font
  {.} %punctuation after thm head
  {.5em} %space after thm head: " " = normal interword space;
  {} %thm head spec (can be left empty, meaning `normal')

\newtheoremstyle{newplain}
  {5pt}
  {5pt}
  {\itshape}
  {}
  {\rmfamily\scshape}
  {. ---}
  {.5em}
  {}

\newtheoremstyle{newremark}
  {5pt}
  {5pt}
  {\rmfamily}
  {}
  {\rmfamily\scshape}
  {. ---}
  {.5em}
  {}

\theoremstyle{newplain}
\newtheorem*{Theorem*}{Theorem} %no numbering for Theorem*

\theoremstyle{newplain}
\newtheorem{Theorem}{Theorem}

\newtheorem{Corollary}[Theorem]{Corollary}
\newtheorem{Proposition}[Theorem]{Proposition}
\newtheorem{Conjecture}[Theorem]{Conjecture}
\newtheorem{Definition}[Theorem]{Definition}

\theoremstyle{newremark}
\newtheorem{Empty}[Theorem]{}
\newtheorem{Remark}[Theorem]{Remark}

\newtheorem{Claim}[Theorem]{Claim}

\theoremstyle{exercise}

\numberwithin{Theorem}{section}
\numberwithin{Exercise}{section}

\newcommand{\ind}{\mathbbm{1}} %indicatrix function

\newcommand{\R}{\mathbf{R}} %real numbers
\newcommand{\Rm}{\R^m}
\newcommand{\Rn}{\R^n}

\renewcommand{\setminus}{\thicksim} %set theoretic difference à la Federer

\newcommand{\calA}{\mathscr{A}}
\newcommand{\calB}{\mathscr{B}}
\newcommand{\calC}{\mathscr{C}}

\newcommand{\calE}{\mathscr{E}}

\newcommand{\calH}{\mathscr{H}}

\newcommand{\calK}{\mathscr{K}}
\newcommand{\calL}{\mathscr{L}}

\newcommand{\calV}{\mathscr{V}}

\newcommand{\calY}{\mathscr{Y}}
\newcommand{\calZ}{\mathscr{Z}}

\newcommand{\frH}{\frak H}

\newcommand{\balpha}{\boldsymbol{\alpha}}

\newcommand{\bdelta}{\boldsymbol{\delta}}

\newcommand{\bxi}{\boldsymbol{\xi}}

\newcommand{\bB}{\mathbf{B}}
\newcommand{\bC}{\mathbf{C}}
\newcommand{\bE}{\mathbf{E}}

\newcommand{\bG}{\mathbf{G}}
\newcommand{\bN}{\mathbf{N}}

\newcommand{\bU}{\mathbf{U}}
\newcommand{\bV}{\mathbf{V}}
\newcommand{\bW}{\mathbf{W}}

\newcommand{\bY}{\mathbf{Y}}

\newcommand{\bc}{\mathbf{c}}

\newcommand{\bv}{\mathbf{v}}
\newcommand{\bw}{\mathbf{w}}

\DeclareMathOperator{\rmBdry}{\mathrm{Bdry}} %boundary
\DeclareMathOperator{\rmClos}{\mathrm{Clos}} %closure
\DeclareMathOperator{\rmdiam}{\mathrm{diam}} %diameter

\DeclareMathOperator{\rmgraph}{\mathrm{graph}} %graph
\DeclareMathOperator{\rmHom}{\mathrm{Hom}} %homomorphisms
\DeclareMathOperator{\rmid}{\mathrm{id}} %the identity map
\DeclareMathOperator{\rmim}{\mathrm{im}} %image
\DeclareMathOperator{\rmInt}{\mathrm{Int}} %interior
\DeclareMathOperator{\rmLip}{\mathrm{Lip}} %Lipschitz constant

\DeclareMathOperator{\rmrank}{\mathrm{rank}}
\DeclareMathOperator{\rmspan}{\mathrm{span}} %span
\DeclareMathOperator{\rmspt}{\mathrm{spt}} %support
\DeclareMathOperator{\rmtrace}{\mathrm{trace}} %trace

\newcommand{\rmI}{\mathrm{I}}
\newcommand{\rmII}{\mathrm{II}}

\newcommand{\lseg}{\boldsymbol{[}\!\boldsymbol{[}}
\newcommand{\rseg}{\boldsymbol{]}\!\boldsymbol{]}}

\newcommand{\hel} {
\hskip2.5pt{\vrule height7pt width.5pt depth0pt}
\hskip-.2pt\vbox{\hrule height.5pt width7pt depth0pt}
\, }

\newcommand{\bin}[2]{
\begin{pmatrix} #1 \\
#2 
\end{pmatrix}}

\def\Xint#1{\mathchoice
{\XXint\displaystyle\textstyle{#1}}%
{\XXint\textstyle\scriptstyle{#1}}%
{\XXint\scriptstyle\scriptscriptstyle{#1}}%
{\XXint\scriptscriptstyle
\scriptscriptstyle{#1}}%
\!\int}
\def\XXint#1#2#3{{%
\setbox0=\hbox{$#1{#2#3}{\int}$}
\vcenter{\hbox{$#2#3$}}\kern-.5\wd0}}

\def\dashint{\Xint-}

\newcommand{\veps}{\varepsilon}

\newcommand{\vphi}{\varphi}
\newcommand{\wh}{\widehat}
\newcommand{\la}{\langle}
\newcommand{\ra}{\rangle}

\renewcommand{\leq}{\leqslant}
\renewcommand{\geq}{\geqslant}
\renewcommand{\subset}{\subseteq}

\newlength{\drop}

\usepackage{trajan}

\begin{document}

%=================
% TITLE AND AUTHOR
%=================

%\titleAT %see above / page de garde d'un livre

\title[Lebesgue null sets]{On Lebesgue Null Sets}

\def\curraddrname{{\itshape On leave of absence from}}

\author[Th. De Pauw]{Thierry De Pauw}
\address{School of Mathematical Sciences\\
Shanghai Key Laboratory of PMMP\\ 
East China Normal University\\
500 Dongchuang Road\\
Shanghai 200062\\
P.R. of China\\
and NYU-ECNU Institute of Mathematical Sciences at NYU Shanghai\\
3663 Zhongshan Road North\\
Shanghai 200062\\
China}
\curraddr{Universit\'e Paris Diderot\\ 
Sorbonne Universit\'e\\
CNRS\\ 
Institut de Math\'ematiques de Jussieu -- Paris Rive Gauche, IMJ-PRG\\
F-75013, Paris\\
France}
\email{thdepauw@math.ecnu.edu.cn,thierry.de-pauw@imj-prg.fr}


%; Secondary 52A40,49J45}

\thanks{The first author was partially supported by the Science and Technology Commission of Shanghai (No. 18dz2271000).}

%\date{December 9th, 2018}

%=========
% ABSTRACT
%=========

\begin{abstract}
Letting $A \subset \Rn$ be Borel and $\bW_0 : \Rn \to \bG(n,m)$ be Lipschitz we establish that $\calL^n(A)=0$ if and only if $\calH^m(A \cap (x+\bW_0(x))=0$ for $\calL^n$ almost every $x \in \Rn$.
\end{abstract}

\maketitle

%===========
% DEDICATION
%===========

\begin{comment}
\begin{flushright}
\textsf{\textit{Dedicated to: Joseph Plateau, Arno, Hooverphonic, Hercule
Poirot,\\
Kim Clijsters, Pierre Rapsat, and non Jef t'es pas tout seul}}
\end{flushright}
\end{comment}

%==================
% TABLE OF CONTENTS
%==================

\tableofcontents
%\newpage

%==============================
% THE MEAT --- OR SO ONE THINKS
%==============================

\section{Foreword}

Let $A$ be a subset of Euclidean space $\Rn$, $n \geq 2$, and let $\calL^n$ denote the Lebesgue outer measure. We concern ourselves with the following question: Can one tell whether $A$ is Lebesgue negligible from the knowledge only of its trace on each member of some given collection of <<lower dimensional>> subsets $\Gamma_i \subset \Rn$, $i \in I$. Thus one expects that if $A \cap \Gamma_i$ is <<negligible in the dimension of $\Gamma_i$>>, for each $i \in I$ then $\calL^n(A)=0$. Of course a necessary condition is that the sets $\Gamma_i$ cover almost all of $A$, i.e. $\calL^n(A \setminus \cup_{i \in I} \Gamma_i)=0$. Consider for instance $n=2$, $I = \R$ and $\Gamma_t = \{t\} \times \R$, $t \in \R$, the collection of all vertical lines in the plane. It is not true in general that if $A \subset \R^2$ and $A \cap \Gamma_t$ is a singleton for each $t \in \R$ then $\calL^2(A)=0$. There exist indeed functions $f : \R \to \R$ whose graph $A = \rmgraph f$ has $\calL^2(A) > 0$, see e.g. \cite[Chapter 2 Theorem 4]{KHARAZISHVILI} for an example due to \textsc{W. Sierpi\'nski}. In order to rule out such examples we will henceforth assume that $A \subset \Rn$ be Borel measurable. In that case the Theorem of \textsc{G. Fubini}, together with the invariance of the Lebesgue measure under orthogonal transformations imply the following. Given an integer $1 \leq m \leq n-1$, if $(\Gamma_i)_{i \in I}$ is the collection of all $m$ dimensional affine subspaces of $\Rn$ of some fixed direction, and if $\calH^m(A \cap \Gamma_i)=0$ for all $i \in I$ then $\calL^n(A)=0$. Here $\calH^m$ denotes the $m$ dimensional Hausdorff measure. A special feature of this collection $(\Gamma_i)_{i \in I}$ is that it partitions $\Rn$, its members being the level sets $f^{-1}\{y\}$, $y \in \R^{n-m}$, of a <<nice map>> $f : \Rn \to \R^{n-m}$, indeed an orthogonal projection. This is an occurrence of the following more general situation when $f$ and its leaves $f^{-1}\{y\}$ are allowed to be nonlinear. The coarea formula due to \textsc{H. Federer} in \cite{FED.59} asserts that if $f : \Rn \to \R^{n-m}$ is Lipschitz and if $A \subset \Rn$ is Borel then
\begin{equation*}
\int_A Jf(x) d\calL^n(x) = \int_{\R^{n-m}} \calH^m\left(A \cap f^{-1}\{y \}\right) d \calL^{n-m}(y) \,.
\end{equation*} 
Thus if the Jacobian coarea factor $Jf$ is positive $\calL^n$ almost everywhere in $A$ then the collection $\left( f^{-1}\{y\} \right)_{y \in \R^{n-m}}$ is suitable for detecting whether or not $A$ is Lebesgue null. At $\calL^n$ almost all $x \in \Rn$ the map $f$ is differentiable according to \textsc{H. Rademacher}, and
\begin{equation*}
Jf(x) = \sqrt{ \left| \det \left( Df(x) \circ Df(x)^*\right)\right|} = \left\| \wedge_{n-m} Df(x) \right\|
\end{equation*}
see \cite[Chapter 3 \S 4]{EVANS.GARIEPY} and \cite[3.2.1 and 3.2.11]{GMT}.
\par 
In this paper we focus on the case when $\Gamma_i$, $i \in I$, are affine subspaces of $\Rn$, but not necessarily members of a partition of the ambient space. Specifically, we assume that with each $x \in \Rn$ is associated an $m$ dimensional affine subspace $\bW(x)$ of $\Rn$ containing $x$. Given a Borel set $A \in \Rn$, the question whether
\begin{equation}
\label{eq.1}
\textit{If } \calH^m(A \cap \bW(x))=0 \text{ for all } x \in A \textit{ then } \calL^n(A)=0 \,,
\end{equation}  
has a negative answer: \textsc{O. Nikod\'ym} \cite{NIK.27} exhibited a Borel subset $A \subset \R^2$ of the unit square, such that $\calL^2(A)=1$ and for each $x \in A$ there exists a line $\bW(x) \subset \R^2$ with the property that $A \cap \bW(x) = \{x\}$. In this context a selection Theorem due to \textsc{J. von Neumann} implies that (possibly considering a smaller, non Lebesgue null Borel subset of $A$) the correspondence $x \mapsto \bW(x)$ can be chosen to be Borel measurable (see \ref{nik.set}) and in turn, it can be chosen to be continuous according to a result of \textsc{N. Lusin}. This was noted by \textsc{A. Zygmund} in connection with multiparameter Fourier analysis. 
\par 
Our result assumes that $\bW$ be Lipschitz. Below $\bG(n,m)$ denotes the Grassmannian manifold of $m$ dimensional linear subspaces of $\Rn$.

\begin{Theorem*}
Assume $\bW_0 : \Rn \to \bG(n,m)$ is Lipschitz and $A \subset \Rn$ is Borel. The following are equivalent.
\begin{enumerate}
\item[(A)] $\calL^n(A)=0$;
\item[(B)] For $\calL^n$ almost every $x \in A$, $\calH^m\left(A \cap (x+\bW_0(x))\right)=0$;
\item[(C)] For $\calL^n$ almost every $x \in \Rn$, $\calH^m\left(A \cap (x+\bW_0(x))\right)=0$.
\end{enumerate}
\end{Theorem*}

This seems to be new. As should be apparent from the discussion above, the difficulty stands with the fact that the affine $m$ planes $\bW(x) = x + \bW_0(x)$ need not be disjointed. The natural route is to reduce the problem to applying the coarea formula by spreading out the $\bW(x)$'s in a disjointed way, in a higher dimensional space, i.e. adding a variable $u \in \bW(x)$ to the given $x \in \Rn$ and considering $\bW(x)$ as a fiber above the base space $\Rn$. We thus define
\begin{equation*}
\Sigma = \Rn \times \Rn \cap \{ (x,u) : x \in E \text{ and } u \in \bW(x) \} \,,
\end{equation*}
where $E \subset \Rn$ is Borel. This set is $n+m$ rectifiable owing to the Lipschitz continuity of $\bW$. It is convenient to assume that $\calL^n(E) < \infty$ so that
\begin{equation}
\label{eq.01}
\phi_E(B) = \int_E \calH^m \left( B \cap \bW(x) \right) d\calL^n(x) \,,
\end{equation}
$B \subset \Rn$, is a locally finite Borel measure, \ref{def.phi}. Now $\Sigma$ was precisely set up so that for each $x \in E$
\begin{equation*}
\calH^m \left( \Sigma \cap \pi_2^{-1}(B) \cap \pi_1^{-1}\{x\} \right) = \calH^m \left( B \cap \bW(x)\right) \,,
\end{equation*}
where $\pi_1$ and $\pi_2$ denote the projections of $\Rn \times \Rn$ to the $x$ and $u$ variable, respectively. Abbreviating $\Sigma_B =  \Sigma \cap \pi_2^{-1}(B)$ the coarea formula yields
\begin{equation}
\label{eq.02}
\phi_E(B) = \int_{\Sigma_B} J_\Sigma \pi_1 d\calH^{n+m} \,.
\end{equation}
A simple calculation shows that $J_\Sigma \pi_2 > 0$ $\calH^{n+m}$ almost everywhere on $\Sigma_B$, \ref{factor.pi.2}. Since also 
\begin{equation}
\label{eq.03}
\int_{\Sigma_B} J_\Sigma \pi_2 d \calH^{n+m} = \int_B \calH^m \left( \Sigma_B \cap \pi_2^{-1}\{u\}\right)d\calL^n(u)
\end{equation}
the implication $(A) \Rightarrow (C)$ above should now be clear: Letting $B=A$ and $E=\bB(0,R)$ one infers from hypothesis (A) and \eqref{eq.03} that $\calH^{n+m}(\Sigma_B) = 0$, thus $\phi_{\bB(0,R)}(A)=0$ from \eqref{eq.02} and in turn conclusion (C) from \eqref{eq.01}.
\par 
In order to establish that $(B) \Rightarrow (A)$ we need to observe that $J_\Sigma \pi_1 > 0$ $\calH^{n+m}$ almost everywhere, \ref{factor.pi.1} and ideally to show that $\calH^m \left( \Sigma_B \cap \pi_2^{-1}\{u\}\right) > 0$ for almost every $u \in B$. This last part offers some difficulty. To understand this we let $m=n-1$ in order to keep the notations short. Now $u \in \bW(x)$ iff $u-x \in \bW_0(x)$ iff $\la \bv_0(x) , x-u \ra =0$ where $\bv_0(x) \in \bW_0(x)^\perp$ is, say a unit vector. Abbreviating $g_u(x) = \la \bv_0(x) , x-u \ra$ we infer that
\begin{equation*}
\calH^m \left( \Sigma_B \cap \pi_2^{-1}\{u\}\right) = \calH^m \left( E \cap g_u^{-1} \{0\}\right)\,.
\end{equation*}
The problem remains that two of the nonlinear $m$ sets $E \cap g_u^{-1} \{0\}$ and $E \cap g_{u'}^{-1} \{0\}$ may intersect, thereby preventing another application of the coarea formula to look out for their lower bound. Yet we already know that
\begin{equation*}
\phi_E(B) = \int_B \calZ_E \bW d\calL^n
\end{equation*}
where $\calZ_E \bW$ is a Radon-Nikod\'ym derivative, \ref{def.Z} and also that $(\calZ_E \bW)(u)$ is comparable to $\calH^m \left( E \cap g_u^{-1} \{0\}\right)$, \ref{Z.1}. Adding an extra variable $y$ to the fibered space $\Sigma$, \ref{fibration.2} we improve on this by showing that
\begin{equation*}
(\calZ_E \bW)(u) \geq \beta_n \liminf_j \dashint_{-j^{-1}}^{j^{-1}} \calH^m \left( E \cap g_u^{-1} \{y\}\right) d\calL^1(y) = \beta_n (\calY_E \bW)(u)\,,
\end{equation*}
where the last equality defines $\calY_E \bW$, and $\beta_n > 0$, \ref{def.Y} and \ref{lb.2}. We are reduced to showing that $\calY_E \bW > 0$ almost everywhere. The reason why this holds is the following. Fix a Borel set $Z \subset \Rn$, $x_0 \in \Rn$ and $r > 0$. Let $\bC_\bW(x_0,r)$ denote the cylindrical box consisting of those $x \in \Rn$ such that $\left|P_{\bW_0(x_0)}(x-x_0) \right| \leq r$ and $\left|P_{\bW_0(x_0)^\perp}(x-x_0) \right| \leq r$. We want to find a lower bound for
\begin{equation*}
\int_{Z \cap \bC_\bW(x_0,r)} \left(\calY_{Z \cap \bC_\bW(x_0,r)}\right)(u) d\calL^n(u)\,.
\end{equation*}
To this end we fix $z \in \bW_0(x_0) \cap \bB(0,r)$ and we let $V_z = \Rn \cap \{ x_0 + z + s\bv_0(x_0) : -r \leq s \leq r \}$ denote the corresponding vertical line segment. According to Fubini's Theorem we are reduced to estimating
\begin{equation*}
\int_{V_z} \left(\calY_{Z \cap \bC_\bW(x_0,r)}\right)(u) d\calH^1(u) \,.
\end{equation*}
According to Vitali's Covering Theorem we can find a disjointed family of line segments $I_1,I_2,\ldots$ covering almost all $V_z$ such that the above integral nearly equals
\begin{multline*}
\sum_k \calH^1(I_k) \dashint_{I_k} \calH^m \left( Z \cap \bC_\bW(x_0,r) \cap g_{u_k}^{-1} \{y\}\right) d\calL^1(y) \\
\cong \sum_k \int_{ Z \cap \bC_\bW(x_0,r) \cap g_{u_k}^{-1} (I_k) } \left| \nabla g_{u_k}(x) \right| d\calL^n(x) \cong \calL^n \left( Z \cap \bC_\bW(x_0,r)\right)
\end{multline*}
where there first near equality follows from the coarea formula, the second one because $\left| \nabla g_{u_k} \right| \cong 1$ at small scales, \ref{jac.g} and the <<nonlinear horizontal stripes>> $g_{u_k}^{-1} (I_k)$ are nearly pairwise disjoint. Verification of these claims takes up sections 5 and 6. Now we reach a contradiction if $Z = \Rn \cap \{ \calY_E \bW = 0\}$ is assumed to have $\calL^n(Z) > 0$ and $x_0$ is a point of density of $Z$.
\par 
I extend my warm thanks to \textsc{Jean-Christophe Léger} for carefully reading the manuscript.

\section{Preliminaries}

\begin{Empty}
In this paper $1 \leq m \leq n-1$ are integers. The ambient space is $\Rn$. The canonical inner product of $x,x' \in \Rn$ is denoted $\la x,x' \ra$ and the corresponding Euclidean norm of $x$ is $|x|$. If $S \subset \Rn$ we let $\calB(S)$ denote the $\sigma$ algebra of Borel subsets of $S$.
\end{Empty}

\begin{Empty}[Hausdorff measure]
\label{h.m}
We let $\calL^n$ denote the Lebesgue outer measure in $\Rn$ and $\balpha(n) = \calL^n(\bB(0,1))$. For $S \subset \Rn$ we abbreviate $\zeta^m(S) = \balpha(m) 2^{-m}(\rmdiam S)^m$. Given $0 < \delta \leq \infty$ we call $\delta$ cover of $A \subset \Rn$ a finite or countable family $(S_j)_{j \in J}$ of subsets of $\Rn$ such that $A \subset \cup_{j \in J} S_j$ and $\rmdiam S_j \leq \delta$ for every $j \in J$. We define
\begin{equation*}
\calH^m_\delta(A) = \inf \left\{ \sum_{j \in J} \zeta^m(S_j) : (S_j)_{j \in J} \text{ is a $\delta$ cover of $A$ }\right\}
\end{equation*}
and $\calH^m(A) = \lim_{\delta \to 0^+} \calH^m_\delta(A) = \sup_{\delta > 0} \calH^m_\delta(A)$. Thus $\calH^m$ is the $m$ dimensional Hausdorff outer measure in $\Rn$. 
\par 
{\it 
\begin{enumerate}
\item[(1)] If $(K_k)_k$ is a sequence of nonempty compact subsets of $\Rn$ converging in Hausdorff distance to $K$ then $\calH^m_\infty(K) \geq \limsup_k \calH^m_\infty(K_k)$.
\end{enumerate}
}
\par 
Given $\veps > 0$ choose a cover $(S_j)_{j=1,2,\ldots}$ of $K$ such that $\calH^m_\infty(K) + \veps \geq \sum_j \zeta^m(S_j)$. Since $\lim_{r \to 0^+} \zeta^m(\bU(S_j,r))=\zeta^m(S_j)$ for each $j=1,2,\ldots$ we can choose an open set $U_j$ containing $S_j$ such that $\zeta^m(U_j) \leq  \veps 2^{-j} + \zeta^m(S_j)$. Since $U = \cup_j U_j$ is open there exists an integer $k_0$ such that $K_k \subset U$ whenever $k \geq k_0$. Thus in that case $(U_j)_j$ is a cover of $K_k$ and therefore $\calH^m_\infty(K_k) \leq \sum_j \zeta^m(U_j) \leq 2\veps + \calH^m_\infty(K)$. Take the $\limsup$ of the left hand side as $k \to \infty$, and then let $\veps \to 0$.
\par 
{\it 
\begin{enumerate}
\item[(2)] For all $A \subset \Rm$ one has $\calL^m(A) = \calH^m(A) = \calH^m_\infty(A)$.
\end{enumerate}
}
\par
It suffices to note that $\calH^m(A) \geq \calH^m_\infty(A) \geq \calL^m(A) \geq \calH^m(A)$. The first inequality is trivial; the second one follows from the isodiametric inequality \cite[2.10.33]{GMT}; the last one is a consequence of the Vitali Covering Theorem \cite[Chapter 2 \S 2 Theorem 2]{EVANS.GARIEPY}.
\par 
{\it 
\begin{enumerate}
\item[(3)] If $W \subset \Rn$ is an $m$ dimensional affine subspace and $A \subset W$ then $\calH^m(A) = \calH^m_\infty(A)$.
\end{enumerate}
}
\par
Let $\frH^m$ denote the $m$ dimensional Hausdorff outer measures in the metric space $W$. In other words 
\begin{equation*}
\frH^m_\delta(A) = \inf \left\{ \sum_{j \in J} \zeta^m(S_j) : (S_j)_{j \in J} \text{ is a $\delta$ cover of $A$ and $S_j \subset W$ for all $j \in J$}\right\} \,,
\end{equation*}
and $\frH^m(A) = \sup_{\delta > 0} \frH^m_\delta(A)$.
It is elementary to observe that $\frH^m(A) = \calH^m(A)$ and that $\frH^m_\infty(A) = \calH^m_\infty(A)$. Now if $f : W \to \Rm$ is an isometry then $\calH^m(A) = \frH^m(A) =  \calH^m(f(A)) = \calH^m_\infty(f(A)) = \frH^m_\infty(A) = \calH^m_\infty(A)$, where the third equality follows from claim (2) above.
\end{Empty}

\begin{Empty}[Coarea formula]
\label{coarea}
Here we recall two versions of the coarea formula. First if $A \subset \Rn$ is $\calL^n$ measurable and $f : A \to \R^{n-m}$ is Lipschitz then $\R^{n-m} \to [0,\infty] : y \mapsto \calH^m \left( A \cap f^{-1}\{y\} \right)$ is $\calL^{n-m}$ measurable and 
\begin{equation*}
\int_A Jf(x) d\calL^n(x) = \int_{\R^{n-m}} \calH^m\left(A \cap f^{-1}\{y \}\right) d \calL^{n-m}(y) \,.
\end{equation*} 
Here the coarea Jacobian factor is well defined $\calL^n$ almost everywhere according to Rademacher's Theorem and equals
\begin{equation*}
Jf(x) = \sqrt{ \left| \det \left( Df(x) \circ Df(x)^*\right)\right|} = \left\| \wedge_{n-m} Df(x) \right\| \,,
\end{equation*}
see for instance \cite[Chapter 3 \S 4]{EVANS.GARIEPY}. 
\par 
Secondly if $A \subset \R^p$ is $\calH^n$ measurable and countably $(\calH^n,n)$ rectifiable, and if $f : A \to \R^{n-m}$ is Lipschitz then $\R^{n-m} \to [0,\infty] : y \mapsto \calH^m \left( A \cap f^{-1}\{y\} \right)$ is $\calL^{n-m}$ measurable and
\begin{equation*}
\int_A J_Af(x) d\calH^n(x) = \int_{\R^{n-m}} \calH^m\left(A \cap f^{-1}\{y \}\right) d \calL^{n-m}(y) \,.
\end{equation*} 
To give a formula for the coarea Jacobian factor $J_Af(x)$ of $f$ relative to $A$ we consider a point $x \in A$ where $A$ admits an approximate $n$ dimensional tangent space $T_xA$ and where $f$ is differentiable along $A$. Letting $L : T_x A \to \R^{n-m}$ denote the derivative of $f$ at $x$ we have
\begin{equation*}
J_Af(x) = \sqrt{ \left| \det \left( L \circ L^*\right)\right|} = \left\| \wedge_{n-m} L \right\| \,,
\end{equation*}
see for instance \cite[3.2.22]{GMT}.
\par 
In both cases it is useful to recall the following. If $L : V \to V'$ is a linear map between two inner product spaces $V$ and $V'$ then
\begin{equation}
\label{eq.11}
\|\wedge_k L \| = \sup \left\{ \la \wedge_k L , \xi \ra : \xi \in \wedge_k V \text{ and } |\xi| = 1 \right\} \,.
\end{equation}
On the one hand $\|\wedge_k L\| \leq \|L\|^k$ \cite[1.7.6]{GMT}, and $\|L\| \leq \rmLip f$ with $L$ as above. On the other hand if $v_1,\ldots,v_k$ are linearly independent vectors of $V$ then 
\begin{equation}
\label{eq.12}
\|\wedge_k L \| \geq \frac{|L(v_1)\wedge \ldots \wedge L(v_k)|}{|v_1 \wedge \ldots \wedge v_k|} \,.
\end{equation}
\par 
Finally we observe that both coarea {formul\ae} hold true when $f$ is merely locally Lipschitz, according to the Monotone Convergence Theorem.
\end{Empty}

\begin{Empty}[Grassmannian]
\label{grassmann}
We let $\bG(n,m)$ denote the set whose members are the $m$ dimensional linear subspaces of $\Rn$. With $W \in \bG(n,m)$ we associate $P_W : \Rn \to \Rn$ the orthogonal projection onto $W$. We give $\bG(n,m)$ the structure of a compact metric space by letting $d(W_1,W_2) = \|P_{W_1}-P_{W_2}\|$. If $W \in \bG(n,m)$ then $W^\perp \in \bG(n-m)$ is so that $P_W + P_{W^\perp}=\rmid_{\Rn}$, therefore $\bG(n,m) \to \bG(n,n-m) : W \mapsto W^\perp$ is an isometry. The bijective correspondence $\vphi :\bG(n,m) \to \rmHom(\Rn,\Rn) : W \mapsto P_W$ identifies $\bG(n,m)$ with the submanifold $M_{n,m} = \rmHom(\Rn,\Rn) \cap \{ L : L \circ L = L\,,\, L^*=L \text{ and } \rmtrace L = m \}$. There exists an open neighborhood $V$ of $M_{n,m}$ in $\rmHom(\Rn,\Rn)$ and a Lipschitz retraction $\rho : V \to M_{n,m}$, according for instance to \cite[3.1.20]{GMT}. Therefore if $S \subset \Rn$ and if $\bW_0 : S \to \bG(n,m)$ is Lipschitz then there exist an open neighborhood $U$ of $E$ in $\Rn$ and a Lipschitz extension $\wh{\bW}_0 : U \to \bG(n,m)$ of $\bW_0$. Indeed $\vphi \circ \bW_0$ admits a Lipschitz extension $\bY : \Rn \to \rmHom(\Rn,\Rn)$, see e.g. \cite[2.10.43]{GMT}, and it suffices to let $U = \bY^{-1}(V)$ and $\wh{\bW}_0 = \rho \circ \left( \bY |_U \right)$.
\end{Empty}

\begin{Empty}[Orthonormal frames]
\label{stiefel}
We let $\bV(n,m)$ denote the set orthonormal $m$ frames in $\Rn$, i.e. $\bV(n,m) = (\Rn)^{m} \cap\{ (w_1,\ldots,w_{m}) : \text{ the family } w_1,\ldots,w_{m} \text{ is orthonormal}\}$. We will consider it as a metric space with its structure inherited from $(\Rn)^{m}$. 
\end{Empty}

\begin{Empty}
\label{grass.param.stief}
{\it 
Let $\calV \subset \bG(n,m)$ be a nonempty closed set such that $\rmdiam \calV < 1$. There exists a Lipschitz map $\Xi : \calV  \to \bV(n,m)$ such that $W = \rmspan \{ \Xi_1(W),\ldots,\Xi_m(W)\}$ for every $W \in \calV$. 
}
\end{Empty}

\begin{proof}
Pick arbitrarily $W_0 \in \calV$. If $W \in \calV$ then the map $W_0 \to W : w \mapsto P_W(w)$ is bijective: if $w \in W_0 \setminus \{0\}$ then $\left|P_W(w)-w\right| = \left|P_W(w)-P_{W_0}(w)\right| < |w|$ thus $P_W(w) \neq 0$. Letting $w_1,\ldots,w_m$ be an arbitrary basis of $W_0$ it follows that for each $W \in \calV$ the vectors $w_i(W) = P_W(w_i)$, $i=1,\ldots,m$, constitute a basis of $W$. Furthermore the maps $w_i : \calV \to \Rn$ are Lipschitz: $\left| w_i(W)-w_i(W')\right| = \left|P_W(w_i) - P_{W'}(w_i)\right| \leq d(W,W') |w_i|$. We apply the Gram-Schmidt process:
\begin{equation*}
\overline{w}_1(W) = w_1(W) \quad\text{ and }\quad \overline{w}_i(W) = w_i(W) - \sum_{j=1}^{i-1} \la w_i(W) , \overline{w}_j(W) \ra \overline{w}_j(W) \,,\,\, i=2,\ldots,m,
\end{equation*}
so that $\overline{w}_1(W),\ldots,\overline{w}_m(W)$ is readily an orthogonal basis of $W$ depending upon $W$ in a Lipschitz way. Since each $\left|\overline{w}_i\right|$ is bounded away from zero on $\calV$ the formula $\Xi_i(W) = \left|\overline{w}_i(W)\right|^{-1}\overline{w}_i(W)$, $i=1,\ldots,m$, defines $\Xi$ with the required property.
\end{proof}

\begin{Empty}
\label{grass.param.stief.borel}
{\it 
There exists a Borel measurable map $\Xi : \bG(n,m)  \to \bV(n,m)$ with the property that $W = \rmspan \{ \Xi_1(W),\ldots,\Xi_m(W)\}$ for every $W \in \bG(n,m)$. 
}
\end{Empty}

\begin{proof}
Since $\bG(n,m)$ is compact it can partitioned into finitely many Borel sets $\calV_1,\ldots,\calV_J$ each having diameter bounded by $1/2$. Define $\Xi$ piecewise to coincide on $\calV_j$ with a $\Xi_j$ associated with $\rmClos \calV_j$ in \ref{grass.param.stief}, $j=1,\ldots,J$.
\end{proof}

\begin{Empty}
\label{orth.frame}
{\it 
Assume $S \subset \Rn$, $x_0 \in S$ and $\bW_0 : S \to \bG(n,m)$ is Lipschitz. There then exist an open neighbordhood $U$ of $x_0$ in $\Rn$ and Lipschitz maps $\bw_1,\ldots,\bw_m,\bv_1,\ldots,\bv_{n-m} : U \to \Rn$ such that:
\begin{enumerate}
\item[(1)] For every $x \in U$ the family $\bw_1(x),\ldots,\bw_m(x),\bv_1(x),\ldots,\bv_{n-m}(x)$ is an orthonormal basis of $\Rn$;
\item[(2)] For every $x \in S \cap U$ one has
\begin{equation*}
\bW_0(x) = \rmspan \{ \bw_1(x),\ldots,\bw_m(x) \}
\end{equation*}
and
\begin{equation*}
\bW_0(x)^\perp = \rmspan \{ \bv_1(x),\ldots,\bv_{n-m}(x) \} \,.
\end{equation*}
\end{enumerate}
}
\end{Empty}

\begin{proof}
We let $\wh{\bW}_0 : \wh{U} \to \bG(n,m)$ be a Lipschitz extension of $\bW_0$ where $\wh{U}$ is an open neighborhood of $S$ in $\Rn$ (recall \ref{grassmann}). Abbreviate $W_0 := \bW_0(x_0)$. Define $\calV = \bG(n,m) \cap \left\{ W : d(W,W_0) < 1/4 \right\}$ and $V = \wh{U} \cap \wh{\bW}_0^{-1}(\calV)$. Apply \ref{grass.param.stief} to $\rmClos \calV$ and denote $\Xi$ the resulting Lipschitz map $\calV \to (\Rn)^m$. Next define $\calV^\perp = \bG(n,n-m) \cap \left\{ W^\perp : W \in \calV \right\}$, apply \ref{grass.param.stief} to $\rmClos \calV^\perp$ and denote $\Xi^\perp$ the resulting Lipschitz map $\calV^\perp \to (\Rn)^{n-m}$. Letting $\bw_i(x) = \left( \Xi_i \circ \wh{\bW}_0 \right)(x)$, $i=1,\ldots,m$, and $\bv_i(x) = \left( \Xi^\perp_i \circ \wh{\bW}_0 \right)(x)$, $i=1,\ldots,n-m$, completes the proof.
\end{proof}

\begin{Empty}
\label{orth.frame.borel}
{\it 
Assume $\bW_0 : \Rn \to \bG(n,m)$ is Borel measurable. There then exist Borel measurable maps $\bw_1,\ldots,\bw_m,\bv_1,\ldots,\bv_{n-m} : \Rn \to \Rn$ such that:
\begin{enumerate}
\item[(1)] For every $x \in \Rn$ the family $\bw_1(x),\ldots,\bw_m(x),\bv_1(x),\ldots,\bv_{n-m}(x)$ is an orthonormal basis of $\Rn$;
\item[(2)] For every $x \in\Rn$ one has
\begin{equation*}
\bW_0(x) = \rmspan \{ \bw_1(x),\ldots,\bw_m(x) \}
\end{equation*}
and
\begin{equation*}
\bW_0(x)^\perp = \rmspan \{ \bv_1(x),\ldots,\bv_{n-m}(x) \} \,.
\end{equation*}
\end{enumerate}
}
\end{Empty}

\begin{proof}
Choose $\Xi : \bG(n,m) \to \bV(n,m)$ and $\Xi^\perp : \bG(n,n-m) \to \bV(n,n-m)$ be as in \ref{grass.param.stief.borel}. Letting $\left(\bw_1(x),\ldots,\bw_m(x)\right) = \left( \Xi \circ \bW_0 \right)(x)$ and $\left(\bv_1(x),\ldots,\bv_{n-m}(x)\right) = \left( \Xi^\perp \circ \bW_0^\perp \right)(x)$, $x \in \Rn$, completes the proof.
\end{proof}

\begin{Empty}[Definition of $\bW(x)$]
\label{W}
The typical situation that arises in the remaining part of this paper is that we are given a set $S \subset \Rn$, a Lipschitz map $\bW_0 : S \to \bG(n,m)$ and $x_0 \in S$. We will represent $\bW_0(x)$ and $\bW_0^\perp(x)$ in a neighborhood $U$ of $x_0$ as in \ref{orth.frame}. We will then further reduce the size of $U$ several times in order that various conditions be met. With no exception we will denote as $\bW(x) = x + \bW_0(x)$ the affine subspace containing $x$, of direction $\bW_0(x)$, whenever $\bW_0(x)$ is defined.
\end{Empty}

\begin{Empty}[Definition of $g_{\bv_1,\ldots,\bv_{n-m},u}$ and lower bound of its coarea factor]
Given an open set $U \subset \Rn$, a Lipschitz map $\bv : U \to \Rn$, and $u \in \Rn$ we define $g_{\bv,u} : U \to \R$ by the formula
\begin{equation*}
g_{\bv,u}(x) = \la \bv(x) , x-u  \ra \,.
\end{equation*} 
Clearly $g_{\bv,u}$ is locally Lipschitz. If $\bv$ is differentiable at $x \in U$ then so is $g_{\bv,u}$ and for every $h \in \Rn$ one has
\begin{equation}
\label{eq.2}
Dg_{\bv,u}(x)(h) = \la \nabla g_{\bv,u}(x),h \ra = \la D\bv(x)(h) , x-u \ra + \la \bv(x) , h \ra \,.
\end{equation}
\par 
Next we assume we are given Lipschitz maps $\bv_1,\ldots,\bv_{n-m} : U \to \Rn$. We define $g_{\bv_1,\ldots,\bv_{n-m},u} : U \to \R^{n-m}$ by the formula
\begin{equation*}
g_{\bv_1,\ldots,\bv_{n-m},u}(x) = \left( g_{\bv_1,u}(x) , \ldots , g_{\bv_{n-m},u}(x) \right) \,.
\end{equation*}
It is Lipschitz as well. The relevance of $g_{\bv_1,\ldots,\bv_{n-m},u}$ stems from the following observation, assuming that $\bv_1,\ldots,\bv_{n-m}$ are associated with $\bW_0$ and $\bW$ as in \ref{orth.frame} and \ref{W}:
\begin{equation}
\label{eq.7}
\begin{split}
u \in \bW(x) & \iff u-x \in \bW_0(x) \\
& \iff \la \bv_i(x) , u-x \ra = 0 \text{ for all } i=1,\ldots,n-m \\
& \iff g_{\bv_1,\ldots,\bv_{n-m},u}(x)=0 \\
& \iff x \in g_{\bv_1,\ldots,\bv_{n-m},u}^{-1}\{0\} \,.
\end{split}
\end{equation}
In fact $|g_{\bv_1,\ldots,\bv_{n-m},u}(x)| = \left| P_{\bW_0(x)^\perp}(x-u)\right|$.
\par 
Abbreviate $g = g_{\bv_1,\ldots,\bv_{n-m},u}$. If each $\bv_i$ is differentiable at $x \in U$, and $h \in \Rn$, then
\begin{equation*}
Dg(x)(h) = \sum_{i=1}^{n-m} Dg_{\bv_i,u}(x)(h)e_i \,,
\end{equation*}
where here and elsewhere $e_1,\ldots,e_{n-m}$ denotes the canonical basis of $\R^{n-m}$.
Thus if $\bv_1(x),\ldots,\bv_{n-m}(x)$ constitute an orthonormal family in $\Rn$ then
\begin{equation*}
Dg_{\bv_i,u}(x)(\bv_j(x)) = \delta_{i,j} + \veps_{i,j}(x,u)
\end{equation*}
where
\begin{equation}
\label{eq.3}
|\veps_{i,j}(x,u)| = \left| \la D\bv_i(x)(\bv_j(x)) , x-u \ra\right| \leq \left( \rmLip \bv_i \right) |x-u| \,,
\end{equation}
according to \eqref{eq.2}, and in turn
\begin{equation*}
Dg(x)(\bv_j(x)) = \sum_{i=1}^{n-m} \left(\delta_{i,j} + \veps_{i,j}(x,u)\right) e_i \,.
\end{equation*}
This allows for a lower bound of the coarea factor of $g$ at $x$ as follows. 
\begin{equation*}
\begin{split}
\left\| \wedge_{n-m}Dg(x) \right\| & \geq \left| Dg(x)(\bv_1(x)) \wedge \ldots \wedge Dg(x)(\bv_{n-m}(x)) \right| \\
& = \left| \left( \sum_{i=1}^{n-m} \left(\delta_{i,1} + \veps_{i,1}(x,u)\right) e_i \right) \wedge \ldots \wedge  \left( \sum_{i=1}^{n-m} \left(\delta_{i,n-m} + \veps_{i,n-m}(x,u)\right) e_i \right) \right| \\
& = \left| \det \left( \delta_{i,j} + \veps_{i,j}(x,u) \right)_{i,j=1,\ldots,n-m} \right| \,.
\end{split}
\end{equation*}
In view of \eqref{eq.3} we obtain the next lemma.
\end{Empty}

\begin{Empty}
\label{jac.g}
{\it 
Given $\Lambda > 0$ and $0 < \veps < 1$ there exists $\bdelta_{\theTheorem}(n,\Lambda,\veps) > 0$ with the following property. Assume that
\begin{enumerate}
\item[(1)] $U \subset \Rn$ is open and $u \in \Rn$;
\item[(2)] $\bv_1,\ldots,\bv_{n-m} : U \to \Rn$ are Lipschitz;
\item[(3)] $\bv_1(x),\ldots,\bv_{n-m}(x)$ is an orthonormal family for every $x \in U$.
\end{enumerate}
If
\begin{enumerate}
\item[(4)] $\rmLip \bv_i \leq \Lambda$ for each $i=1,\ldots,n-m$;
\item[(5)] $\rmdiam \left( U \cup \{u\} \right) \leq \bdelta_{\theTheorem}(n,\Lambda,\veps)$
\end{enumerate}
then
\begin{equation*}
Jg_{\bv_1,\ldots,\bv_{n-m},u}(x) \geq 1 - \veps
\end{equation*}
at $\calL^n$ almost every $x \in U$.
}
\end{Empty}

\begin{Empty}[Definition of $\pi_u$ and its relation with $g_{\bv_1,\ldots,\bv_{n-m},u}$]
\label{pi.u}
With $u \in \Rn$ we associate
\begin{equation*}
\pi_u : \bV(n,n-m) \times \Rn \to \R^{n-m} : (\xi_1,\ldots,\xi_{n-m},x) \mapsto\left( \la \xi_1,x-u\ra, \ldots, \la \xi_{n-m},x-u\ra \right) \,.
\end{equation*}
When $(\xi_1,\ldots,\xi_{n-m}) \in \bV(n,n-m)$ is fixed we also abbreviate as $\pi_{\xi_1,\ldots,\xi_{n-m},u}$ the map $\Rn \to \R^{n-m}$ defined by $\pi_{\xi_1,\ldots,\xi_{n-m},u}(x) = \pi_u(\xi_1,\ldots,\xi_{n-m},x)$. It is then rather useful to observe that in the context described in \ref{orth.frame} and \ref{W} the following holds:
\begin{equation}
\label{eq.4}
\pi_{\bv_1(x),\ldots,\bv_{n-m}(x),u}^{-1}\left\{ g_{\bv_1,\ldots,\bv_{n-m},u}(x) \right\} = \bW(x) \,.
\end{equation}
Indeed,
\begin{equation*}
\begin{split}
h \in \bW(x) & \iff h-x \in \bW_0(x) \\
& \iff \la \bv_i(x) , h-x \ra = 0 \text{ for all }  i=1,\ldots,n-m \\
& \iff \la \bv_i(x) , h-u \ra = \la \bv_i(x) , x-u \ra \text{ for all } i=1,\ldots,n-m \\
& \iff \pi_{\bv_1(x),\ldots,\bv_{n-m}(x),u}(h) = g_{\bv_1,\ldots,\bv_{n-m},u}(x) \,.
\end{split}
\end{equation*}
In the sequel we will sometimes abbreviate $\xi = (\xi_1,\ldots,\xi_{n-m}) \in \bV(n,n-m)$.
It also helps to notice that for given $\xi \in \bV(n,n-m)$ and $y \in \R^{n-m}$ the set $\pi_{\xi,u}^{-1}\{y\}$ is an $m$ dimensional affine subspace of $\Rn$.
\end{Empty}

\begin{Empty}
\label{measurability}
{\it 
Assume $B \in \calB(\Rn)$ and $u \in \Rn$. It follows that
\begin{equation*}
h_{B} :  \bV(n,n-m) \times \R^{n-m} \to \R_+ : (\xi,y) \mapsto \calH^m\left(B  \cap \pi_{\xi,u}^{-1}\{y\} \right) 
\end{equation*}
is Borel measurable.
}
\end{Empty}

\begin{proof}
We start by showing that when $B$ is compact, $h_{B}$ is upper semicontinuous. Thus if $(\xi_k,y_k) \in \bV(n,n-m) \times \R^{n-m}$ converge to $(\xi,y)$, we ought to show that
\begin{equation}
\label{eq.5}
\calH^m_\infty(K) \geq \limsup_k \calH^m_\infty(K_k)
\end{equation}
where $K = B  \cap \pi_{\xi,u}^{-1}\{y\}$ and $K_k = B  \cap \pi_{\xi_k,u}^{-1}\{y_k\}$. This is indeed equivalent to the same inequality with $\calH^m_\infty$ replaced by $\calH^m$ according to \ref{h.m}(3) and the last sentence of \ref{pi.u}. Considering if necessary a subsequence of $(K_k)_k$ we may assume that none of the compact sets $K_k$ is empty, and that the $\limsup$ in \eqref{eq.5} is a $\lim$. Since the set of nonempty compact subsets of the compact set $B$, equipped with the Hausdorff metric is compact, the sequence $(K_k)_k$ admits a subsequence (denoted the same way) converging to a compact set $L \subset B$. Given $z \in L$ there are $z_k \in K_k$ converging to $z$. Thus  $\pi_u(\xi,z) = \lim_k \pi_u(\xi_k,z_k) = \lim_k y_k = y$. In other words $z \in K$. Thus $\calH^m_\infty(K) \geq \calH^m_\infty(L)$ and \eqref{eq.5} follows from \ref{h.m}(1).
\par 
Next we abbreviate $\calA = \calB(\Rn) \cap \{ B : h_{B} \text{ is Borel measurable}\}$. Thus we have just shown that $\calA$ contains the collection $\calK(\Rn)$ of all compact subsets of $\Rn$. Observe that if $(B_j)_j$ is an increasing sequence in $\calA$ and $B = \cup_j B_j$ then $h_{B} = \lim_j h_{B_j}$ pointwise, thus $B \in \calA$. In particular $\Rn \in \calA$. Finally if $B,B' \in \calA$ and $B' \subset B$ then $h_{B \setminus B'} = h_{B} - h_{B'}$ because all measures involved are finite, indeed $h_{B}(\xi,y) \leq \balpha(m)r^m$ for all $(\xi,y)$. Accordingly $B \setminus B' \in \calA$. This means that $\calA$ is a Dynkin class. Since $\calK(\Rn)$ is a $\pi$ system, $\calA$ contains the $\sigma$ algebra generated by $\calK(\Rn)$, i.e. $\calB(\Rn)$, \cite[Theorem 1.6.2]{COHN}.
\end{proof}

\begin{Empty}
\label{second.measurability}
{\it 
Assume $B \in \calB(\Rn)$, $r > 0$ and $\bW_0 : \Rn \to \bG(n,m)$ is Borel measurable. The following function is Borel measurable.
\begin{equation*}
\Rn \to [0,\infty] : x \mapsto \calH^m\left( B \cap \bW(x) \right)\,.
\end{equation*}
}
\end{Empty}

\begin{proof}
Let $h_{\bW,B}$ denote this function. Let $\bv_1,\ldots,\bv_{n-m} : \Rn \to \Rn$ be Borel measurable maps associated with $\bW_0$ as in \ref{orth.frame.borel}. Fix $u \in \Rn$ arbitrarily. Define
\begin{equation*}
\Upsilon : \Rn \to \bV(n,n-m) \times \R^{n-m} : x \mapsto \left(\bv_1(x),\ldots,\bv_{n-m}(x),g_{\bv_1(x),\ldots,\bv_{n-m}(x),u}(x) \right)
\end{equation*}
so that
\begin{equation*}
h_{\bW,B} = h_{B} \circ \Upsilon
\end{equation*}
(where $h_{B}$ is the function associated with $B$ and $u$ in \ref{measurability}), according to \eqref{eq.4}. One notes that $\Upsilon$ is Borel measurable, and the conclusion ensues from \ref{measurability}.
\end{proof}

\begin{Empty}[Definition of $\phi_{E,\bW}$]
\label{def.phi}
Let $\bW_0 : \Rn \to \bG(n,m)$ be Borel measurable and let $E \in \calB(\Rn)$ be such that $\calL^n(E) < \infty$. For each $B \in \calB(\Rn)$ we define
\begin{equation*}
\phi_{E,\bW}(B) = \int_E \calH^m \left( B \cap \bW(x) \right) d\calL^n(x) \,.
\end{equation*}
This is well defined according to \ref{second.measurability}(2). It is easy to check that $\phi_{E,\bW}$ is a locally finite (hence $\sigma$ finite) Borel measure on $\Rn$; indeed $\phi_{E,\bW}(B) \leq \balpha(m) (\rmdiam B)^m \calL^n(E)$. 
\par 
To close this section we discuss the relevance of $\phi_{E,\bW}$ to the problem of existence of <<nearly Nikod\'ym sets>>.
\end{Empty}

\begin{Empty}[Definition of Nearly Nikod\'ym set]
Let $E \in \calB(\Rn)$. We say that $B \in \calB(E)$ is {\bf nearly $m$ Nikod\'ym in $E$} if
\begin{enumerate}
\item[(1)] $\calL^n(B) > 0$;
\item[(2)] For $\calL^n$ almost each $x \in E$ there is $W \in \bG(n,m)$ such that $\calH^m\left( B \cap (x+W) \right) = 0$.
\end{enumerate}
In case $n=2$, $m=1$, $E=[0,1] \times [0,1]$, the existence of such $B$ (with $\calL^2(B)=1$) was established by \textsc{O. Nikod\'ym} \cite{NIK.27}, see also \cite[Chapter 8]{GUZMAN.1981}. For arbitrary $n \geq 2$ and $m=n-1$ the existence of such $B$ was established by \textsc{K. Falconer} \cite{FAL.86}. In fact in both cases these authors proved the stronger condition that for every $x \in B$, $\calH^m(B \cap (x+W))=0$ can be replaced by $B \cap (x+W) = \{x\}$. Thus in case $1 \leq m< n-1$, if $B$ is a set exhibited by \textsc{K. Falconer}, $x \in B$ and $W \subset \in \bG(n,n-1)$ is such that $B \cap (x+W) = \{x\}$, picking arbitrarily $V \in \bG(n,m)$ such that $V \subset W$ we see that $B \cap (x+V)=\{x\}$. Whence $B$ is also nearly $m$ Nikod\'ym in $B$. 
\par 
Assuming also that $\bW_0 : E \to \bG(n,m)$ is Borel measurable we say that $B \in \calB(E)$ is {\bf nearly $m$ Nikod\'ym in $E$ relative to $\bW$} if
\begin{enumerate}
\item[(1)] $\calL^n(B) > 0$;
\item[(2)] For $\calL^n$ almost each $x \in E$ one has $\calH^m\left( B \cap \bW(x) \right) = 0$.
\end{enumerate}
\end{Empty}

\begin{Empty}
{\it 
Let $E \in \calB(\Rn)$ and let $\bW_0 : \Rn \to \bG(n,m)$ be Borel measurable. The following are equivalent.
\begin{enumerate}
\item[(1)] $\calL^n |_{\calB(E)}$ is absolutely continuous with respect to $\phi_{E,\bW}|_{\calB(E)}$.
\item[(2)] There does not exist a nearly $m$ Nikod\'ym set relative to $\bW$.
\end{enumerate}
}
\end{Empty}

\begin{proof}
A set $B \in \calB(E)$ such that $\phi_{E,\bW}(B)=0$ and $\calL^n(B) > 0$ is, by definition a nearly $m$ Nikod\'ym set relative to $\bW$. Condition (1) is equivalent to their nonexistence.
\end{proof}

\begin{Empty}
\label{nik.set}
{\it 
Assume that $E \in \calB(\Rn)$ and that $B \in \calB(E)$ is nearly $m$ Nikod\'ym. It follows that:
\begin{enumerate}
\item[(1)] There exists $\bW_0 : \Rn \to \bG(n,m)$ Borel measurable such that $B$ is nearly $m$ Nikod\'ym in $E$ relative to $\bW$.
\item[(2)] There exists $C \subset B$ compact and $\bW_0 : \Rn \to \bG(n,m)$ continuous such that $C$ is nearly $m$ Nikod\'ym in $C$ relative to $\bW$.
\end{enumerate}
}
\end{Empty}

\begin{proof}
Define a Borel measurable map $\bxi : \bG(n,m) \to \bV(n-m)$ by $\bxi(W) = \Xi \left( W^\perp \right)$ where $\Xi : \bG(n,n-m) \to \bV(n,n-m)$ is as in \ref{grass.param.stief.borel}. Choose arbitrarily $u \in \Rn$ and define a Borel measurable map
\begin{multline*}
\Upsilon :  E \times \bG(n,m) \to  \bV(n,n-m) \times \R^{n-m} \\ (x,W) \mapsto \left(\bxi(W) , \la \bxi_1(W),x-u  \ra , \ldots, \la \bxi_{n-m}(W),x-u \ra \right) \,.
\end{multline*} 
Similarly to \eqref{eq.4} observe that
\begin{equation*}
W = \pi_{\bxi(W),u}^{-1} \left\{ \left( \la \bxi_1(W),x-u  \ra , \ldots, \la \bxi_{n-m}(W),x-u \ra \right) \right\}
\end{equation*}
for every $W \in \bG(n,m)$. We infer from \ref{measurability} that
\begin{equation*}
 h_{B} \circ \Upsilon :E \times \bG(n,m) \to [0,\infty]:(x,W) \mapsto \calH^m \left( B  \cap (x+W) \right)
\end{equation*}
is Borel measurable. Thus the set
\begin{equation*}
\calE = E \times \bG(n,m) \cap \left\{ (x,W) : \calH^m \left( B \cap (x+W) \right) = 0 \right\}
\end{equation*}
is Borel as well. The set $N = E \cap \{ x : \calE_x = \emptyset \}$ is coanalytic and $\calL^n(N)=0$ by assumption. According to von Neumann's selection Theorem \cite[5.5.3]{SRIVASTAVA} there exists a universally measurable map $\tilde{\bW}_0 : E \setminus N \to \bG(n,m)$ such that $\tilde{\bW}_0(x) \in \calE_x$ for every $x \in E \setminus N$, i.e. $\calH^m \left( B \cap \left( x + \tilde{\bW}_0(x) \right) \right) = 0$. We extend $\tilde{\bW}_0$ to be an arbitrary constant on $N$. This makes $\tilde{\bW}_0$ an $\calL^n$ measurable map defined on $E$. Therefore it is equal $\calL^n$ almost everywhere to a Borel map $\bW_0 : E \to \bG(n,m)$. This proves (1).
\par 
In order to prove (2) we recall of \ref{grassmann}, specifically the retraction $\rho : V \to M_{n,m}$ and the homeomorphic identification $\vphi : \bG(n,m) \to M_{n,m}$. Owing to the compactness of $M_{n,m}$ there are finitely many open balls $U_j$, $j=1,\ldots,J$, whose closure are contained in $V$ and covering $M_{n,m}$. Since $\calL^n(B) > 0$ there exists $j=1,\ldots,J$ such that $\calL^n \left( B \cap E_j \right) > 0$ where $E_j = \left( \vphi \circ \bW_0 \right)^{-1}(U_j)$. It follows from Lusin's Theorem \cite[2.5.3]{GMT} that there exists a compact set $C \subset B \cap E_j$ such that $\calL^n(C) > 0$ and the restriction $\bW_0|_C$ is continuous. The map $\vphi \circ \bW_0|_C$ takes its values in the closed ball $\rmClos U_j$, therefore admits a continuous extension $\bY : \Rn \to \rmClos U_j \subset V$. Letting $\bW = \vphi^{-1} \circ \rho \circ \bY$ completes the proof.
\end{proof}

\section{Common setting}

\begin{Empty}[Setting for the next three sections]
\label{31}
In the next three sections we shall assume the following.
\begin{enumerate}
\item $E \subset \Rn$ is Borel and $\calL^n(E) < \infty$.
\item $U \subset \Rn$ is open and $E \subset U$.
\item $B \subset \Rn$ is Borel.
\item $\bW_0 : E \to \bG(n,m)$ is Lipschitz.
\item $\bW(x) = x + \bW_0(x)$ for each $x \in E$.
\item $\Lambda > 0$.
\item $\bw_1,\ldots,\bw_m : U \to \Rn$ and $\rmLip \bw_i \leq \Lambda$, $i=1,\ldots,m$.
\item $\bv_1,\ldots,\bv_{n-m} : U \to \Rn$ and $\rmLip \bv_i \leq \Lambda$, $i=1,\ldots,n-m$.
\item $\bW_0(x) = \rmspan \{ \bw_1(x),\ldots,\bw_m(x)\}$ for every $x \in E$.
\item $\bW_0(x)^\perp = \rmspan \{ \bv_1(x),\ldots,\bv_{n-m}(x)\}$ for every $x \in E$.
\item $\bw_1(x),\ldots,\bw_m(x),\bv_1(x),\ldots,\bv_{n-m}(x)$ constitute an orthonormal basis of $\Rn$, for every $x \in E$.
\end{enumerate}
\end{Empty}

\section{Two fibrations}

\begin{Empty}[A fibered space associated with $E,B,\bw_1,\ldots,\bw_m$]
\label{fibration.1}
We define 
\begin{equation*}
F : E \times \Rm \to \Rn \times \Rn : (x,t_1,\ldots,t_m) \mapsto \left( x , x + \sum_{i=1}^m t_i \bw_i(x) \right)
\end{equation*}
as well as 
\begin{equation*}
\Sigma = F( E \times \Rm) = \Rn \times \Rn \cap \left\{ (x,u) : x \in E \text{ and } u \in \bW(x) \right\}\,.
\end{equation*}
 It is obvious that $F$ is Lipschitz and therefore $\Sigma$ is countably $n+m$ rectifiable and $\calH^{n+m}$ measurable. We also consider the two canonical projections
\begin{equation*}
\pi_1 : \Rn \times \Rn \to \Rn : (x,u) \mapsto x \quad\text{ and }\quad \pi_2 : \Rn \times \Rn \to \Rn : (x,u) \mapsto u
\end{equation*}
as well as 
\begin{equation*}
\Sigma_B = \Sigma \cap \pi_2^{-1}(B) = \Rn \times \Rn \cap \left\{ (x,u) : x \in E \text{ and } u \in B \cap \bW(x) \right\}\,,
\end{equation*}
which is clearly also countably $n+m$ rectifiable and $\calH^{n+m}$ measurable. In view of applying the coarea formula to $\Sigma_B$ and $\pi_1$ first, to $\Sigma_B$ and $\pi_2$ next, we observe that
\begin{equation*}
\Sigma_B \cap \pi_1^{-1}\{x\} = \Rn \times \Rn \cap \{ (x,u) : u \in B \cap \bW(x) \}
\end{equation*}
so that
\begin{equation}
\label{eq.6}
\calH^m \left(\Sigma_B \cap \pi_1^{-1}\{x\} \right) = \calH^m \left( B \cap \bW(x)\right)
\end{equation}
whenever $x \in E$, and that
\begin{equation*}
\begin{split}
\Sigma_B \cap \pi_2^{-1}\{u\} & = \Rn \times \Rn \cap \{ (x,u) : x \in E \text{ and } u \in \bW(x) \} \\
& = \Rn \times \Rn \cap \left\{ (x,u) : x \in E \cap g_{\bv_1,\ldots,\bv_{n-m},u}^{-1}\{0\} \right\}
\end{split}
\end{equation*}
according to \eqref{eq.7}, so that
\begin{equation}
\label{eq.8}
\calH^m \left( \Sigma_B \cap \pi_2^{-1}\{u\}\right) = \calH^m \left( E \cap g_{\bv_1,\ldots,\bv_{n-m},u}^{-1}\{0\}\right) 
\end{equation}
whenever $u \in B$. It now follows from the coarea formula that
\begin{equation}
\label{eq.9}
\int_{\Sigma_B} J_\Sigma \pi_1 d\calH^{n+m} = \int_E \calH^m \left( B \cap \bW(x) \right) d\calL^n(x) = \phi_{E,\bW}(B)
\end{equation}
and
\begin{equation}
\label{eq.10}
\int_{\Sigma_B} J_\Sigma \pi_2 d\calH^{n+m} = \int_B \calH^m \left( E \cap g_{\bv_1,\ldots,\bv_{n-m},u}^{-1}\{0\}\right) d\calL^n(u) \,.
\end{equation}
For these formul\ae{} to be useful we need to establish bounds for the coarea Jacobian factors $J_\Sigma \pi_1$ and $J_\Sigma \pi_2$. In order to do so we notice that if $\Sigma \ni (x,u) = F(x,t_1,\ldots,t_m)$ and if $F$ is differentiable at $(x,t_1,\ldots,t_m)$ then the approximate tangent space $T_{(x,u)}\Sigma$ exists and is generated by the following $n+m$ vectors of $\Rn \times \Rn$:
\begin{equation*}
\begin{split}
\frac{\partial F}{\partial x_j}(x,t) & = \left( e_j, e_j + \sum_{i=1}^m t_i \frac{\partial \bw_i}{\partial x_j}(x) \right) \,, \,j=1,\ldots,n \\
\frac{\partial F}{\partial t_k}(x,t) & = \left( 0 , \bw_k(x) \right) \,, \,k=1,\ldots,m \,.
\end{split}
\end{equation*}
As usual $e_1,\ldots,e_n$ denotes the canonical basis of $\Rn$.
\end{Empty}

\begin{Empty}[Coarea Jacobian factor of $\pi_1$] 
\label{factor.pi.1}
{\it 
For $\calH^{n+m}$ almost every $(x,u) \in \Sigma$ one has
\begin{equation*}
\left( 2 + 2m\Lambda |x-u| + m^2\Lambda^2 |x-u|^2 \right)^{-\frac{n}{2}} \leq J_\Sigma \pi_1(x,u) \leq 1 \,.
\end{equation*}
}
\end{Empty}

\begin{proof}
We recall \ref{coarea}.
The right hand inequality follows from $\rmLip \pi_1 = 1$. Regarding the left hand inequality fix $(x,u) = F(x,t)$ such that $F$ is differentiable at $(x,t)$ and let $L : T_{(x,u)}\Sigma \to \R^n$ denote the restriction of $\pi_1$ to $T_{(x,u)}\Sigma$. Put $v_j = \frac{\partial F}{\partial x_j}(x,t)$, $j=1,\ldots,n$, and recall \eqref{eq.12} that
\begin{equation*}
J_\Sigma \pi_1(x,u) = \| \wedge_n L \| \geq \frac{1}{|v_1 \wedge \ldots \wedge v_n|}
\end{equation*}
since $L(v_j)=e_j$, $j=1\ldots,n$. Now notice that 
\begin{multline*}
\left| \frac{\partial F}{\partial x_j}(x,t)\right|^2 = |e_j|^2 + \left| e_j + \sum_{i=1}^m t_i \frac{\partial \bw_i}{\partial x_j}(x)\right|^2 
\leq 2 + 2 \left|\sum_{i=1}^m t_i \frac{\partial \bw_i}{\partial x_j}(x) \right| + \left|\sum_{i=1}^m t_i \frac{\partial \bw_i}{\partial x_j}(x) \right|^2 \\
\leq 2 + 2 m \Lambda |t| + m^2\Lambda^2|t|^2 \,.
\end{multline*}
Since $u = x + \sum_{i=1}^m t_i \bw_i(x)$ one also has
\begin{equation*}
|u-x|^2 = \left| \sum_{i=1}^m t_i \bw_i(x) \right|^2 = |t|^2 \,.
\end{equation*}
Finally,
\begin{multline*}
|v_1 \wedge \ldots \wedge v_n | = \left| \frac{\partial F}{\partial x_1}(x,t) \wedge \ldots \wedge \frac{\partial F}{\partial x_1}(x,t) \right| \leq \prod_{j=1}^n \left|  \frac{\partial F}{\partial x_j}(x,t) \right| \\
 \leq \left( 2 + 2 m \Lambda |x-u| + m^2\Lambda^2|x-u|^2 \right)^\frac{n}{2}
\end{multline*}
and the conclusion follows.
\end{proof}

\begin{Empty}
\label{34}
{\it 
Let $1 \leq q \leq n-1$ be an integer and let $v_1,\ldots,v_q$ be an orthonormal family in $\Rn$. There then exists $\lambda \in \Lambda(n,q)$ such that 
\begin{equation*}
\left| \det \left( \la v_k , e_{\lambda(j)} \ra \right)_{j,k=1,\ldots,q}\right| \geq \bin{n}{q}^{-\frac{1}{2}} \,.
\end{equation*}
Here $\Lambda(n,q)$ denotes the set of increasing maps $\{1,\ldots,q\} \to \{1,\ldots,n\}$.
}
\end{Empty}

\begin{proof}
We define a linear map $L : \R^q \to \Rn : (s_1,\ldots,s_q) \mapsto \sum_{k=1}^q s_k v_k$ and we observe that $L$ is an isometry. Therefore its area Jacobian factor $JL=1$, by definition. Now also
\begin{equation*}
(JL)^2 = \sum_{\lambda \in \Lambda(n,q)}\left| \det \left( \la v_k , e_{\lambda(j)} \ra \right)_{j,k=1,\ldots,q}\right|^2
\end{equation*}
according to the Binet-Cauchy formula \cite[Chapter 3 \S 2 Theorem 4]{EVANS.GARIEPY}. The conclusion easily follows.
\end{proof}

\begin{Empty}[Coarea Jacobian factor of $\pi_2$] 
\label{factor.pi.2}
{\it 
The following hold.
\begin{enumerate}
\item 
For $\calH^{n+m}$ almost every $(x,u) \in \Sigma$ one has
\begin{equation*}
\frac{\bin{n}{n-m}^{-\frac{1}{2}} - \left(2^{n-m}-1 \right) m \Lambda |u-x|}{\left( 2 + 2 m \Lambda |x-u| + m^2\Lambda^2|x-u|^2 \right)^\frac{n-m}{2}} \leq J_\Sigma \pi_2(x,u) \leq 1 \,.
\end{equation*}
\item 
For $\calH^{n+m}$ almost every $(x,u) \in \Sigma$ one has $J_\Sigma \pi_2(x,u) > 0$.
\end{enumerate}
}
\end{Empty}

\begin{proof}
Clearly $J_\Sigma \pi_2(x,u) \leq (\rmLip \pi_2)^n \leq 1$. Regarding the left hand inequality fix $(x,u) = F(x,t)$ such that $F$ is approximately differentiable at $(x,t)$ and this time let $L : T_{(x,u)}\Sigma \to \R^n$ denote the restriction of $\pi_2$ to $T_{(x,u)}\Sigma$. We will now define a family of $n$ vectors $v_1,\ldots,v_n$ belonging to $T_{(x,u)}\Sigma$. We choose $v_k = \frac{\partial F}{\partial t_k}(x,t) = (0,\bw_k(x))$ for $k=1,\ldots,m$. For choosing the $n-m$ remaining vectors we proceed as follows. We select $\lambda \in \Lambda(n,n-m)$ as in \ref{34} applied with $q=n-m$ to $\bv_1(x),\ldots,\bv_{n-m}(x)$, and we let $v_{m+j} = \frac{\partial F}{\partial x_{\lambda(j)}}(x,t)$, $j=1,\ldots,n-m$. Recalling \eqref{eq.12} we have
\begin{equation*}
J_\Sigma \pi_1(x,u) = \| \wedge_n L \| \geq \frac{|L(v_1) \wedge \ldots \wedge L(v_n)|}{|v_1 \wedge \ldots \wedge v_n|} \,.
\end{equation*}
As in the proof of \ref{factor.pi.1} we find that
\begin{equation*}
|v_1 \wedge \ldots \wedge v_n| \leq \left( 2 + 2 m \Lambda |x-u| + m^2\Lambda^2|x-u|^2 \right)^\frac{n-m}{2}
\end{equation*}
and it remains only to find a lower bound for $|L(v_1) \wedge \ldots \wedge L(v_n)|$. This equals the absolute value of the determinant of the matrix of coefficients of $L(v_i)$, $i=1,\ldots,n$, with respect to any orthonormal basis of $\Rn$. We choose the basis $\bw_1(x),\ldots,\bw_m(x),\bv_1(x),\ldots,\bv_{n-m}(x)$. Thus
\begin{equation}
\label{eq.13}
\begin{split}
|L(v_1) \wedge \ldots \wedge L(v_n)| & = \left| \det \left(
\begin{array}{c c c | c}
1 & \cdots & 0 & * \\
\vdots & \ddots & \vdots & \vdots \\
0 & \cdots & 1 & * \\
\hline
0 & \cdots & 0 & \left\la e_{\lambda(j)} + \sum_{i=1}^m t_i \frac{\partial \bw_i}{\partial x_{\lambda(j)}}(x) , \bv_k(x) \right\ra \\
\end{array}
\right)\right| \\
& = \left| \det \left( \left\la e_{\lambda(j)} + \sum_{i=1}^m t_i \frac{\partial \bw_i}{\partial x_{\lambda(j)}}(x) , \bv_k(x) \right\ra\right)_{j,k=1,\ldots,n-m}\right| \,.
\end{split}
\end{equation}
Abbreviate
\begin{equation*}
h_{\lambda(j)} = \sum_{i=1}^m t_i \frac{\partial \bw_i}{\partial x_{\lambda(j)}}(x)
\end{equation*}
and observe that $\left| h_{\lambda(j)} \right| \leq m \Lambda |t| = m \Lambda |x-u|$, $j=1,\ldots,n-m$ (recall the proof of \ref{factor.pi.1}). It remains only to remember that $\lambda$ has been selected in order that
\begin{equation*}
\left| \det \left( \left\la e_{\lambda(j)}, \bv_k(x)\right\ra\right)_{j,k=1,\ldots,n-m} \right| \geq \bin{n}{n-m}^{-\frac{1}{2}}
\end{equation*}
and to infer from the multilinearity of the determinant that
\begin{multline*}
\left| \det \left( \left\la e_{\lambda(j)} , \bv_k(x)\right\ra + \left\la h_{\lambda(j)} , \bv_k(x)\right\ra\right) - \det \left( \left\la e_{\lambda(j)} , \bv_k(x)\right\ra\right)\right| \\ \leq \left(2^{n-m}-1 \right) \left(\max_{j=1,\ldots,n-m} \left|\left\la h_{\lambda(j)},\bv_k(x) \right\ra \right| \right)\left( \max_{j,k=1,\ldots,m} \left| \left\la e_{\lambda(j)} , \bv_k(x) \right\ra\right|\right)^{n-m-1}\\
\leq \left(2^{n-m}-1 \right)m\Lambda |x-u| \,.
\end{multline*}
This completes the proof of conclusion (1).
\par 
Let $E_0$ denote the subset of $E$ consisting of those $x$ such that each $\bw_i$, $i=1,\ldots,m$, is differentiable at $x$. Thus $E_0$ is Borel and so is
\begin{multline*}
A = E_0 \times \Rm \cap \big\{ (x,t) : \\ \rmrank 
\left(
\begin{array}{c|c|c|c|c|c}
\bw_1(x) & \cdots & \bw_m(x) & e_1 + \sum_{i=1}^m t_i \frac{\partial \bw_i}{\partial x_1}(x) & \ldots & e_n + \sum_{i=1}^m t_i \frac{\partial \bw_i}{\partial x_n}(x) 
\end{array}
\right) < n
\big\}
\end{multline*}
If $(x,u) \in \Sigma \setminus F(A)$ then the restriction of $\pi_2$ to $T_{(x,u)}\Sigma$ is surjective and therefore $J_\Sigma \pi_2(x,u) > 0$. Thus we ought to show that $\calH^{n+m}(F(A))=0$. Since $F$ is Lipschitz it suffices to establish that $\calL^{n+m}(A)=0$. As $A$ is Borel it is enough to prove that $\calL^m(A_x)=0$ for every $x \in E_0$, according to Fubini's Theorem. Fix $x \in E_0$. As in the proof of conclusion (1), choose $\lambda \in \Lambda(n,n-m)$ associated with $\bv_1(x),\ldots,\bv_{n-m}(x)$ according to \ref{34}. Based on \eqref{eq.13} we see that
\begin{equation*}
A_x \subset \Rm \cap \left\{ t :  \det \left( \left\la e_{\lambda(j)} + \sum_{i=1}^m t_i \frac{\partial \bw_i}{\partial x_{\lambda(j)}}(x) , \bv_k(x) \right\ra\right)_{j,k=1,\ldots,n-m}=0\right\}
\end{equation*}
The set on the right is of the form $S_x = \Rm \cap \{ (t_1,\ldots,t_m ) : P_x(t_1,\ldots,t_m) = 0 \}$ for some polynomial $P_x \in \R[T_1,\ldots,T_m]$, and $P_x(0,\ldots,0) = \det\left( \la e_{\lambda(j)},\bv_k(x) \ra\right)_{j,k=1,\ldots,n-m} \neq 0$. It follows that $\calL^m(S_x)=0$, see e.g. \cite[2.6.5]{GMT} and the proof of (2) is complete.
\end{proof}

\begin{Proposition}
\label{AC.1}
The measure $\phi_{E,\bW}$ is absolutely continuous with respect to $\calL^n$.
\end{Proposition}

\begin{proof}
Let $B \in \calB(\Rn)$ be such that $\calL^n(B)=0$. It follows from \eqref{eq.10} that
\begin{equation*}
\int_{\Sigma_B} J_\Sigma \pi_2 d\calH^{n+m} = 0 \,.
\end{equation*}
It next follows from \ref{factor.pi.2}(2) that $\calH^{n+m}(\Sigma_B)=0$. In turn \eqref{eq.9} implies that
\begin{equation*}
\phi_{E,\bW}(B) = \int_{\Sigma_B} J_\Sigma \pi_1 d\calH^{n+m} = 0 \,.
\end{equation*}
\end{proof}

\begin{Empty}[Definition of $\calZ_E \bW$]
\label{def.Z}
Note that $\phi_{E,\bW}$ is a $\sigma$ finite Borel measure on $\Rn$ (see \ref{def.phi}) and that it is absolutely continuous with respect to $\calL^n$ (see \ref{AC.1}). It then ensues from the Radon-Nikod\'ym Theorem that there exists a Borel measurable function
\begin{equation*}
\calZ_E \bW : E \to \R
\end{equation*}
such that for every $B \in \calB(\Rn)$ one has
\begin{equation*}
\int_E \calH^m \left( B \cap \bW(x) \right) d\calL^n(x) = \phi_{E,\bW}(B) = \int_B \calZ_E \bW(u) d\calL^n(u) \,.
\end{equation*}
Furthermore $\calZ_E \bW$ is univoquely defined only up to a $\calL^n$ null set. This will not affect the reasonings in this paper. Each time we will write $\calZ_E \bW$ we will mean {\it one} particular Borel measurable function verifying the above equality for every $B \in \calB(\Rn)$.
\end{Empty}

\begin{Empty}[Definition of $\calY_E^0 \bW$]
\label{def.Y0}
We define $\calY^0_E\bW : \Rn \to [0,\infty]$ by the formula
\begin{equation}
\label{eq.14}
\calY_E^0 \bW(u) = \calH^m \left( E \cap g_{\bv_1,\ldots,\bv_{n-m},u}^{-1}\{0\}\right) \,
\end{equation}
$u \in \Rn$. Letting $B = \Rn$ in \eqref{eq.8} one infers from \ref{coarea} that $\calY^0_E\bW$ is $\calL^n$ measurable. Using the estimates we have established so far regarding coarea Jacobian factors we now show that $\calZ_E \bW$ and $\calY_E^0 \bW$ are comparable when the diameter of $E$ is not too large.
\end{Empty}

\begin{Proposition}
\label{Z.1}
Given $0 < \veps < 1$ there exists $\bdelta_{\theTheorem}(n,\Lambda,\veps) > 0$ with the following property. If $\rmdiam E \leq \bdelta_{\theTheorem}(n,\Lambda,\veps)$ then
\begin{equation*}
(1-\veps) 2^{-\frac{n}{2}} \calY^0_E \bW(u) \leq \calZ_E \bW(u) \leq (1+\veps) 2^\frac{n-m}{2} \bin{n}{n-m}^\frac{1}{2	}\calY^0_E \bW(u)
\end{equation*}
for $\calL^n$ almost every $u \in E$.
\end{Proposition}

\begin{proof}
We readily infer from \ref{factor.pi.1} and \ref{factor.pi.2}(1) that there exists $\bdelta(n,\Lambda,\veps) > 0$ such that for $\calH^{n+m}$ almost all $(x,u) \in \Sigma$ if $|x-u| \leq \bdelta(n,\Lambda,\veps)$ then 
\begin{equation}
\alpha := (1-\veps) 2^{-\frac{n}{2}} \leq J_\Sigma \pi_1(x,u)
\end{equation}
and
\begin{equation}
\beta := (1+\veps)^{-1} 2^{-\frac{n-m}{2}} \bin{n}{n-m}^{-\frac{1}{2}} \leq J_\Sigma \pi_2(x,u) 
\end{equation}
where the above define $\alpha$ and $\beta$.
\par 
Assume now that $\rmdiam E \leq \bdelta(n,\Lambda,\veps)$. Given $B \in \calB(E)$ we infer from \eqref{eq.9}, \ref{factor.pi.1}, \ref{factor.pi.2}(1), \eqref{eq.10} and the above lower bounds that
\begin{multline*}
\phi_{E,\bW}(B)  = \int_{\Sigma_B} J_\Sigma \pi_1 d\calH^{n+m} 
 \geq \alpha \calH^{n+m}(\Sigma_B) 
 \geq \alpha \int_{\Sigma_B} J_\Sigma \pi_2 d\calH^{n+m} \\
 = \alpha \int_B \calY^0_E \bW d\calL^n
\end{multline*}
and
\begin{multline*}
\phi_{E,\bW}(B) = \int_{\Sigma_B} J_\Sigma \pi_1 d\calH^{n+m} \leq \calH^{n+m}(\Sigma_B) \leq \beta^{-1} \int_{\Sigma_B} J_\Sigma \pi_2 d\calH^{n+m} \\
 = \beta^{-1} \int_B  \calY^0_E \bW d\calL^n \,.
\end{multline*}
Thus
\begin{equation*}
\int_B \beta^{-1} \calY^0_E \bW d\calL^n \leq \int_B \calZ_E \bW d\calL^n \leq \int_B \alpha \calY^0_E \bW d\calL^n
\end{equation*}
for every $B \in \calB(\Rn)$. The conclusion follows from the $\calL^n$ measurability of both $\calZ_E \bW$ and $\calY^0_E \bW$.
\end{proof}

\begin{Empty}[Rest stop]
The above upper bound for $\calZ_E \bW$ is already enough to bound it in turn, by a constant times $(\rmdiam E)^m$, see \ref{cor.ub}. However I would not know how to use the above lower bound to establish that $\calZ_E \bW > 0$ almost everywhere in $E$, which is what we are after. Indeed in the definition \eqref{eq.14} of $\calY^0_E \bW(u)$, $u$ does not appear as the covariable of the function whose level set we are measuring, thereby preventing the use of the coarea formula in an attempt to estimate $\calY^0_E \bW(u)$. This naturally leads to adding a variable $y \in \R^{n-m}$ to the fibered space $\Sigma$, a covariable for $g_{\bv_1,\ldots,\bv_{n-m},u}$.
\end{Empty}

\begin{Empty}[A fibered space associated with $E, B, \bw_1,\ldots,\bw_m,\bv_1,\ldots,\bv_{n-m}$]
\label{fibration.2}
Let $r > 0$, and abbreviate $C_r = \R^{n-m} \cap \left\{ y : |y| \leq r \right\}$ the Euclidean ball centered at the origin, of radius $r$ in $\R^{n-m}$. We define
\begin{equation*}
\hat{F}_r : E \times \Rm \times C_r \to \Rn \times \Rn \times \R^{n-m} : (x,t,y) \mapsto \left(x , x + \sum_{i=1}^m t_i \bw_i(x) + \sum_{i=1}^{n-m} y_i \bv_i(x) , y \right)
\end{equation*}
and
\begin{equation*}
\hat{\Sigma}_r = \hat{F}_r \left( E \times \Rm \times C_r \right) = \Rn \times \Rn \times C_r \cap \left\{ (x,u,y) : x \in E \text{ and } u \in \bW(x) + \sum_{i=1}^{n-m} y_i\bv_i(x)\right\}
\end{equation*}
so that $\hat{F}_r$ is Lipschitz and $\hat{\Sigma}_r$ is countably $2n$ rectifiable and $\calH^{2n}$ measurable. Similarly to \ref{fibration.1} we define
\begin{equation*}
\hat{\Sigma}_{r,B} = \hat{\Sigma}_r \cap \pi_2^{-1}(B) 
\end{equation*}
which clearly is also countably $2n$ rectifiable and $\calH^{2n}$ measurable. We aim to apply the coarea formula to $\hat{\Sigma}_{r,B}$ and to the two projections
\begin{equation*}
\pi_1 \times \pi_3 : \Rn \times \Rn \times \R^{n-m} \to \Rn \times \R^{n-m} : (x,u,y) \mapsto (x,y)
\end{equation*}
and
\begin{equation*}
\pi_2 \times \pi_3 : \Rn \times \Rn \times \R^{n-m} \to \Rn \times \R^{n-m} : (x,u,y) \mapsto (u,y) \,.
\end{equation*}
To this end we notice that
\begin{multline*}
\hat{\Sigma}_{r,B} \cap \left( \pi_1 \times \pi_3 \right)^{-1} \{ (x,y) \} \\ = \Rn \times \Rn \times \R^{n-m} \cap \left\{ (x,u,y) : u \in B \cap \left( \bW(x) + \sum_{i=1}^{n-m} y_i \bv_i(x) \right)\right\}
\end{multline*}
and thus
\begin{equation*}
\calH^m \left( \hat{\Sigma}_{r,B} \cap \left( \pi_1 \times \pi_3 \right)^{-1} \{ (x,y) \}\right) = \calH^m \left(B \cap \left( \bW(x) + \sum_{i=1}^{n-m} y_i \bv_i(x) \right) \right)
\end{equation*}
for every $(x,y) \in E \times C_r$. We further notice that
\begin{equation*}
\begin{split}
\hat{\Sigma}_{r,B} \cap ( \pi_2 & \times \pi_3 )^{-1} \{ (u,y) \} \\ &= \Rn \times \Rn \times \R^{n-m} \cap \left\{ (x,u,y) : x \in E \text{ and } u \in \bW(x) + \sum_{i=1}^{n-m} y_i\bv_i(x)\right\}\\
&= \Rn \times \Rn \times \R^{n-m} \cap \left\{ (x,u,y) : x \in E \cap g_{\bv_1,\ldots,\bv_{n-m},u}^{-1}\{y\} \right\} \,,
\end{split}
\end{equation*}
because
\begin{equation*}
\begin{split}
u \in \bW(x) + \sum_{i=1}^{n-m} y_i \bv_i(x) & \iff u-x- \sum_{i=1}^{n-m} y_i \bv_i(x) \in \bW_0(x) \\
& \iff \left\la \bv_j(x) , u-x- \sum_{i=1}^{n-m} y_i \bv_i(x) \right\ra = 0 \text{ for all } j=1,\ldots,n-m \\
& \iff g_{\bv_1,\ldots,\bv_{n-m},u}(x) = y
\end{split}
\end{equation*}
and therefore
\begin{equation*}
\calH^m \left(\hat{\Sigma}_{r,B} \cap ( \pi_2 \times \pi_3 )^{-1} \{ (u,y) \} \right) = \calH^m \left(E \cap g_{\bv_1,\ldots,\bv_{n-m},u}^{-1}\{y\} \right)
\end{equation*}
whenever $u \in B$ and $y \in C_r$.
\par 
It now follows from the coarea formula and Fubini's Theorem that
\begin{equation}
\label{eq.20}
\int_{\hat{\Sigma}_{r,B}} J_{\hat{\Sigma}_r} (\pi_1 \times \pi_3) d\calH^{2n} = \int_E  d\calL^n(x)\int_{C_r} \calH^m \left(B \cap \left( \bW(x) + \sum_{i=1}^{n-m} y_i \bv_i(x) \right) \right) d\calL^{n-m}(y) 
\end{equation}
and that
\begin{equation}
\label{eq.21}
\int_{\hat{\Sigma}_{r,B}} J_{\hat{\Sigma}_r} (\pi_2 \times \pi_3) d\calH^{2n} = \int_B d\calL^n(u) \int_{C_r} \calH^m \left( E \cap g_{\bv_1,\ldots,\bv_{n-m},u}^{-1}\{y\} \right)d\calL^{n-m}(y) \,.
\end{equation}
\end{Empty}

\begin{Empty}[Coarea Jacobian factors of $\pi_1 \times \pi_3$ and $\pi_2 \times \pi_3$]
\label{factor}
{\it 
The following inequalities hold for $\calH^{2n}$ almost every $(x,u,y) \in \hat{\Sigma}_r$.
\begin{equation*}
2^{-\frac{n-m}{2}} \left( 2 + 4n\Lambda |u-x| + 3n^2 \Lambda^2 |u-x|^2 \right)^{-\frac{n}{2}} \leq J_{\hat{\Sigma}_r} (\pi_1 \times \pi_3)(x,u,y)
\end{equation*}
and 
\begin{equation*}
 J_{\hat{\Sigma}_r} (\pi_2 \times \pi_3)(x,u,y) \leq 1 \,.
\end{equation*}
}
\end{Empty}

\begin{proof}
The second conclusion is obvious since $\rmLip \pi_2 \times \pi_3 = 1$. Regarding the first conclusion we reason similarly as in the proof of \ref{factor.pi.1}. Fix $(x,u,y) = \hat{F}_r(x,t,y)$ such that $\hat{F}_r$ is differentiable at $(x,t,y)$ and denote by $L$ the restriction of $\pi_1 \times \pi_3$ to $T_{(x,u,y)}\hat{\Sigma}_r$. This tangent space is generated by the following $2n$ vectors of $\Rn \times \Rn \times \R^{n-m}$
\begin{equation*}
\begin{split}
\frac{\partial \hat{F}_r}{\partial x_j}(x,t,y) & = \left( e_j , e_j + \sum_{i=1}^m t_i \frac{\partial \bw_i(x)}{\partial x_j}(x) + \sum_{i=1}^{n-m} y_i \frac{\partial \bv_i}{\partial x_j}(x) , 0 \right) \,,\, j=1,\ldots,n \\
\frac{\partial \hat{F}_r}{\partial t_k}(x,t,y) & = \left( 0, \bw_k(x), 0 \right) \,,\, k=1,\ldots,m \\
\frac{\partial \hat{F}_r}{\partial y_\ell}(x,t,y) & = \left( 0, \bv_\ell(x) , e_\ell \right) \,,\, \ell=1,\ldots,n-m \,.
\end{split}
\end{equation*}
The range of $\pi_1 \times \pi_3$ being $2n -m$ dimensional we need to select $2n-m$ vectors $v_1,\ldots,v_{2n-m}$ in $ T_{(x,u,y)}\hat{\Sigma}_r$ to obtain a lower bound
\begin{equation}
\label{eq.22}
J_{\hat{\Sigma}_r} (\pi_1 \times \pi_3)(x,u,y) = \| \wedge_{2n-m} L \| \geq \frac{|L(v_1) \wedge \ldots \wedge L(v_{2n-m}|}{|v_1 \wedge \ldots \wedge v_{2n-m}|} \,.
\end{equation}
The obvious choice consists of $v_j = \frac{\partial \hat{F}_r}{\partial x_j}(x,t,y)$, $j=1,\ldots,n$, and $v_{n+\ell} = \frac{\partial \hat{F}_r}{\partial y_\ell}(x,t,y)$, $\ell=1,\ldots,n-m$, so that $L(v_1),\ldots,L(v_{n-m})$ is the canonical basis of $\Rn \times \R^{n-m}$ and therefore the numerator in \eqref{eq.22} equals 1. In order to determine an upper bound for its denominator we start by fixing $j=1,\ldots,n$, we abbreviate $a_j(x,t,y) = \sum_{i=1}^m t_i \frac{\partial \bw_i(x)}{\partial x_j}(x)$ and $b_j(x,t,y) = \sum_{i=1}^{n-m} y_i \frac{\partial \bv_i}{\partial x_j}(x)$ and we notice that $|a_j(x,t,y)| \leq m \Lambda |t| \leq n \Lambda |t|$, $|b_j(x,t,y)| \leq (n-m) \Lambda |y| \leq n \Lambda |y|$. Furthermore since $u - x = \sum_{i=1}^m t_i \bw_i(x) + \sum_{i=1}^{n-m} y_i \bv_i(x)$ one has $|u-x|^2 = |t|^2 + |y|^2 \leq \max \{|t|^2,|y|^2\}$. Therefore
\begin{equation*}
\begin{split}
\left| \frac{\partial \hat{F}_r}{\partial x_j}(x,t,y)\right|^2 &= |e_j|^2 + |e_j + a_j(x,t,y) + b_j(x,t,y)|^2 \\
& \leq 1 + 1 + |a_j(x,t,y)|^2 + |b_j(x,t,y)|^2 \\
& \quad \quad \quad+ 2|a_j(x,t,y)| + 2 |b_j(x,t,y)| + 2 |a_j(x,t,y)||b_j(x,t,y)| \\
& \leq 2 + 4n\Lambda |u-x| + 3n^2 \Lambda^2 |u-x|^2 \,.
\end{split}
\end{equation*}
Moreover 
\begin{equation*}
\left| \frac{\partial \hat{F}_r}{\partial y_\ell}(x,t,y) \right| = \sqrt{2}
\end{equation*}
for each $\ell=1,\ldots,n-m$. We conclude that
\begin{multline*}
|v_1\wedge \ldots \wedge v_{2n-m} | \leq \left( \prod_{j=1}^n \left| \frac{\partial \hat{F}_r}{\partial x_j}(x,t,y)\right|\right) \left(\prod_{\ell=1}^{n-m} \left| \frac{\partial \hat{F}_r}{\partial y_\ell}(x,t,y) \right| \right) \\
\leq 2^\frac{n-m}{2} \left( 2 + 4n\Lambda |u-x| + 3n^2 \Lambda^2 |u-x|^2 \right)^\frac{n}{2}
\end{multline*}
and the proof is complete.
\end{proof}

\begin{Empty}[Definition of $\calY_E \bW$]
\label{def.Y}
It follows from the Coarea Theorem that the function
\begin{equation*}
\Rn \times \R^{n-m} \to [0,\infty] : (u,y) \to \calH^m \left( E \cap g_{\bv_1,\ldots,\bv_{n-m},u}^{-1} \{y\}\right) 
\end{equation*}
is $\calL^n \otimes \calL^{n-m}$ measurable (recall \ref{fibration.2} applied with $B = \Rn$). It now follows from Fubini's Theorem that for each $r > 0$ the function
\begin{equation*}
\Rn \to [0,\infty] : u \mapsto \dashint_{C_r} \calH^m \left( E \cap g_{\bv_1,\ldots,\bv_{n-m},u}^{-1} \{y\}\right) d\calL^{n-m}(y)
\end{equation*}
is $\calL^n$ measurable. In turn the function
\begin{equation*}
\calY_E \bW : \Rn \to [0,\infty] : u \mapsto \liminf_j \dashint_{C_{j^{-1}}} \calH^m \left( E \cap g_{\bv_1,\ldots,\bv_{n-m},u}^{-1} \{y\}\right) d\calL^{n-m}(y)
\end{equation*}
is $\calL^n$ measurable. It is a replacement for $\calY^0_E \bW$ defined in \ref{def.Y0}. We shall establish for $\calZ_E \bW$ a similar lower bound to that in \ref{Z.1}, this time involving $\calY_E \bW$. Before doing so, we notice the rather trivial fact that if $F \subset E$ then
\begin{equation*}
\calY_F \bW(u) \leq \calY_E \bW(u)
\end{equation*}
for all $u \in \Rn$.
\end{Empty}

\begin{Empty}[Preparatory remark for the proof of \ref{lb.1}]
\label{rem.1}
It follows from the Coarea Theorem that the function
\begin{equation*}
\Rn \times \R^{n-m} \to [0,\infty] : (x,y) \mapsto \calH^m \left(B \cap \left( \bW(x) + \sum_{i=1}^{n-m} y_i \bv_i(x) \right) \right) 
\end{equation*}
is $\calL^n \otimes \calL^{n-m}$ measurable (recall \ref{fibration.2} applied with $B = \Rn$). It therefore follows from Fubini's Theorem as in \ref{def.Y} that
\begin{equation*}
f_j : \Rn \to [0,\infty] : x \mapsto \limsup_j \dashint_{C_{j^{-1}}} \calH^m \left(B \cap \left( \bW(x) + \sum_{i=1}^{n-m} y_i \bv_i(x) \right) \right) d\calL^{n-m}(y)
\end{equation*}
is $\calL^n$ measurable. Furthermore if $B$ is bounded then $|f_j(x)| \leq \balpha(m) (\rmdiam B)^m$ for every $x \in \Rn$. 
\end{Empty}

\begin{Empty}
\label{usc}
{\it 
If $B$ is compact then for every $x \in E$ the function
\begin{equation*}
\R^{n-m} \to \R_+ : y \mapsto \calH^m \left(B \cap \left( \bW(x) + \sum_{i=1}^{n-m} y_i \bv_i(x) \right) \right) 
\end{equation*}
is upper semicontinuous.
}
\end{Empty}

\begin{proof}
The proof is analogous to that of \ref{measurability}.
For each $y \in \R^{n-m}$ define the compact set $K_y = B \cap \left( \bW(x) + \sum_{i=1}^{n-m} y_i \bv_i(x) \right)$. If $(y_k)_k$ is a sequence converging to $y$ we ought to show that
\begin{equation*}
\calH^m_\infty \left(K_y\right) \geq \limsup_k \calH^m_\infty \left( K_{y_k} \right) \,.
\end{equation*}
Since each $K_y$ is a subset of an $m$ dimensional affine subspace of $\Rn$ this is indeed equivalent to the same inequality with $\calH^m_\infty$ replaced by $\calH^m$ according to \ref{h.m}(3). Considering if necessary a subsequence of $(y_k)_k$ we may assume that none of the compact sets $K_{y_k}$ is empty and the the above $\limsup$ is a $\lim$. Considering yet a further subsequence we may now assume that $(K_{y_k})_k$ converges in Hausdorff distance to some compact set $L \subset B$. One checks that $L \subset K_y$. It then follows from \ref{h.m}(1) that $\calH^m_\infty \left(K_y\right) \geq \calH^m_\infty (L) \geq \limsup_k \calH^m_\infty \left( K_{y_k} \right)$.
\end{proof}

\begin{Proposition}
\label{lb.1}
Given $0 < \veps < 1$ there exists $\bdelta_{\theTheorem}(n,\Lambda,\veps) > 0$ with the following property.
If $\rmdiam ( E \cup B) \leq \bdelta_{\theTheorem}(n,\Lambda,\veps)$ and if $B$ is compact then
\begin{equation*}
\int_E \calH^m \left( B \cap \bW(x) \right) d\calL^n(x) \geq (1-\veps) 2^{-\frac{2n-m}{2}} \int_B \calY_E \bW(u) d\calL^n(u) \,.
\end{equation*}
\end{Proposition}

\begin{proof}
We first observe that we can choose $\bdelta_{\ref{lb.1}}(n,\Lambda,\veps) > 0$ small enough so that 
\begin{equation}
\label{eq.23}
J_{\hat{\Sigma}_r} (\pi_1 \times \pi_3)(x,u,y) \geq (1-\veps) 2^{-\frac{2n-m}{2}}
\end{equation}
for $\calH^{2n}$ almost every $(x,u,y) \in \hat{\Sigma}_r$ provided $|u-x| \leq \bdelta_{\ref{lb.1}}(n,\Lambda,\veps)$, according to \ref{factor}. Thus \eqref{eq.23} holds for $\calH^{2n}$ almost every $(x,u,y) \in \hat{\Sigma}_{r,B}$ under the assumption that  $\rmdiam ( E \cup B) \leq \bdelta_{\ref{lb.1}}(n,\Lambda,\veps)$. When \eqref{eq.20}, \eqref{eq.21} and \ref{factor} imply that
\begin{equation}
\label{eq.25}
\begin{split}
\int_E  d\calL^n(x) & \int_{C_r} \calH^m \left(B \cap \left( \bW(x) + \sum_{i=1}^{n-m} y_i \bv_i(x) \right) \right) d\calL^{n-m}(y) \\
&= \int_{\hat{\Sigma}_{r,B}} J_{\hat{\Sigma}_r} (\pi_1 \times \pi_3) d\calH^{2n} \\
& \geq (1-\veps) 2^{-\frac{2n-m}{2}} \calH^{2n} \left( \hat{\Sigma}_{r,B}\right) \\
& \geq (1-\veps) 2^{-\frac{2n-m}{2}}\int_{\hat{\Sigma}_{r,B}} J_{\hat{\Sigma}_r} (\pi_2 \times \pi_3) d\calH^{2n} \\
& = (1-\veps) 2^{-\frac{2n-m}{2}} \int_B d\calL^n(u) \int_{C_r} \calH^m \left( E \cap g_{\bv_1,\ldots,\bv_{n-m},u}^{-1}\{y\} \right)d\calL^{n-m}(y) \,.
\end{split}
\end{equation}
\par 
Fix $x \in E$ and $\beta > 0$. According to \ref{usc} there exists a positive integer $j(x,\beta)$ such that if $j \geq j(x,\beta)$ then
\begin{equation*}
\begin{split}
\calH^m \left( B \cap \bW(x) \right) + \beta & \geq \sup_{y \in C_{j^{-1}}}
\calH^m \left(B \cap \left( \bW(x) + \sum_{i=1}^{n-m} y_i \bv_i(x) \right) \right) \\
& \geq \dashint_{C_{j^{-1}}}\calH^m \left(B \cap \left( \bW(x) + \sum_{i=1}^{n-m} y_i \bv_i(x) \right) \right) d\calL^{n-m}(y) \,.
\end{split}
\end{equation*} 
Taking the $\limsup$ as $j \to \infty$ on the right hand side, and letting $\beta \to 0$ we obtain
\begin{equation}
\label{eq.24}
\calH^m \left( B \cap \bW(x) \right) \geq \limsup_j \dashint_{C_{j^{-1}}}\calH^m \left(B \cap \left( \bW(x) + \sum_{i=1}^{n-m} y_i \bv_i(x) \right) \right) d\calL^{n-m}(y) \,.
\end{equation}
As this holds for all $x \in E$ we may integrate over $E$ with respect to $\calL^n$. Noticing that for every $j=1,2,\ldots$ (with the notation of \ref{rem.1}) $|f_j| \leq \balpha(m) (\rmdiam B)^m \ind_E$, the latter being $\calL^n$ summable, justifies the application of the reverse Fatou lemma below. Thus the following ensues from \eqref{eq.24}, the reverse Fatou lemma, \eqref{eq.25}, and the Fatou lemma:
\begin{equation*}
\begin{split}
\int_E \calH^m &\left( B \cap \bW(x) \right) d\calL^n(x) \\& \geq \int_E d\calL^n(x) \limsup_j \dashint_{C_{j^{-1}}}\calH^m \left(B \cap \left( \bW(x) + \sum_{i=1}^{n-m} y_i \bv_i(x) \right) \right) d\calL^{n-m}(y) \\
& \geq \limsup_j \int_E d\calL^n(x)\dashint_{C_{j^{-1}}}\calH^m \left(B \cap \left( \bW(x) + \sum_{i=1}^{n-m} y_i \bv_i(x) \right) \right) d\calL^{n-m}(y) \\
& \geq (1-\veps) 2^{-\frac{2n-m}{2}} \limsup_j \int_B d\calL^n(u) \dashint_{C_{j^{-1}}} \calH^m \left( E \cap g_{\bv_1,\ldots,\bv_{n-m},u}^{-1}\{y\} \right)d\calL^{n-m}(y)\\
& \geq (1-\veps) 2^{-\frac{2n-m}{2}} \liminf_j \int_B d\calL^n(u) \dashint_{C_{j^{-1}}} \calH^m \left( E \cap g_{\bv_1,\ldots,\bv_{n-m},u}^{-1}\{y\} \right)d\calL^{n-m}(y)\\
& \geq (1-\veps) 2^{-\frac{2n-m}{2}}  \int_B d\calL^n(u)\liminf_j \dashint_{C_{j^{-1}}} \calH^m \left( E \cap g_{\bv_1,\ldots,\bv_{n-m},u}^{-1}\{y\} \right)d\calL^{n-m}(y)\\
&= (1-\veps) 2^{-\frac{2n-m}{2}}\int_B \calY_E \bW(u)d\calL^n(u) \,.
\end{split}
\end{equation*}
\end{proof}

\begin{Corollary}
\label{lb.2}
If $0 < \veps < 1$ and $\rmdiam E \leq \bdelta_{\ref{lb.1}}(n,\Lambda,\veps)$ then 
\begin{equation*}
\calZ_E \bW(u) \geq (1-\veps) 2^{-\frac{2n-m}{2}}\calY_E \bW(u)
\end{equation*}
for $\calL^n$ almost every $u \in E$.
\end{Corollary}

\section{Upper bound for $\calY_E \bW$ and $\calZ_E \bW$}

\begin{Empty}[Bow Tie Lemma]
\label{bow.tie}
{\it 
Let $S \subset \Rn$, $W \in \bG(n,m)$ and $0 < \tau < 1$. Assume that 
\begin{equation*}
(\forall \,x \in S)(\forall \,0 < \rho \leq \rmdiam S) : S \cap \bB(x,\rho) \subset \bB(x+W,\tau\rho) \,.
\end{equation*}
There then exists $F : P_W(S) \to \Rn$ such that $S=\rmim F$ and $\rmLip F \leq \frac{1}{\sqrt{1-\tau^2}}$. In particular
\begin{equation*}
\calH^m(S) \leq \left(\frac{1}{\sqrt{1-\tau^2}} \right)^m \balpha(m) (\rmdiam S)^m \,.
\end{equation*}
}
\end{Empty}

\begin{proof}
Let $x,x' \in S$ and define $\rho = |x-x'| \leq \rmdiam S$. Thus $x' \in S \cap \bB(x,\rho)$ and therefore $\left|P_{W^\perp}(x-x') \right| \leq \tau \rho = \tau |x-x'|$. Since $|x-x'|^2 = \left|P_{W}(x-x') \right|^2 + \left|P_{W^\perp}(x-x') \right|^2$ we infer that
\begin{equation*}
(1-\tau^2) \left|x-x'\right|^2 \leq \left|P_{W}(x-x') \right|^2
\end{equation*}
Therefore $P_W|_S$ is injective, and the Lipschitz bound on $F = \left( P_W|_S \right)^{-1}$ clearly follows from the above inequality. Regarding the second conclusion,
\begin{equation*}
\calH^m(S) = \calH^m\left(F(P_W(S)) \right) \leq \left( \rmLip F\right)^m \calH^m \left(P_W(S) \right),
\end{equation*}
and $P_W(S)$ is contained in a ball of radius $\rmdiam P_W(S) \leq \rmdiam S$.
\end{proof}

\begin{Empty}
\label{ub.1}
{\it 
Given $0 < \tau < 1$ there exists $\bdelta_{\theTheorem}(n,\Lambda,\tau) > 0$ with the following property. If 
\begin{enumerate}
\item $x_0 \in U$ and $u \in \Rn$;
\item $\rmdiam\left( E \cup \{x_0\}\cup\{u\}\right) \leq \bdelta_{\theTheorem}(n,\Lambda,\tau)$;
\end{enumerate}
Then: For every $y \in \R^{n-m}$, for every $x \in E \cap g_{\bv_1,\ldots,\bv_{n-m},u}^{-1}\{y\}$ and for every $0 < \rho < \infty$ one has
\begin{equation*}
E \cap g_{\bv_1,\ldots,\bv_{n-m},u}^{-1}\{y\} \cap \bB(x,\rho) \subset \bB \left( x + \bW_0(x_0),\tau \rho\right)\,.
\end{equation*}
}
\end{Empty}

\begin{proof}
We shall show that $\bdelta_{\ref{ub.1}}(n,\Lambda,\tau)=\frac{\tau}{2\Lambda \sqrt{n}}$ will do. Let $x,x' \in E \cap g_{\bv_1,\ldots,\bv_{n-m},u}^{-1}\{y\}$ for some $y \in \R^{n-m}$. Thus $g_{\bv_1,\ldots,\bv_{n-m},u}(x)=g_{\bv_1,\ldots,\bv_{n-m},u}(x')$ and hence
\begin{multline*}
0 = \left| g_{\bv_1,\ldots,\bv_{n-m},u}(x) - g_{\bv_1,\ldots,\bv_{n-m},u}(x')\right| = \sqrt{\sum_{i=1}^{n-m} \left| \la \bv_i(x),x-u\ra -  \la \bv_i(x'),x'-u\ra\right|^2} \\
= \sqrt{\sum_{i=1}^{n-m} \left| \la \bv_i(x),x-x'\ra - \la \bv_i(x') - \bv_i(x),x'-u \ra\right|^2}\\
\geq \sqrt{\sum_{i=1}^{n-m} \left| \la \bv_i(x),x-x'\ra\right|^2} - \sqrt{\sum_{i=1}^{n-m}\left| \la \bv_i(x') - \bv_i(x),x'-u \ra\right|^2}\,,
\end{multline*}
thus
\begin{multline*}
\sqrt{\sum_{i=1}^{n-m} \left| \la \bv_i(x),x-x'\ra\right|^2} \leq \sqrt{\sum_{i=1}^{n-m}\left| \la \bv_i(x') - \bv_i(x),x'-u \ra\right|^2} \\
\leq \sqrt{n-m} \Lambda |x-x'||x'-u| \leq \frac{\tau}{2}|x-x'|\,.
\end{multline*}
In turn,
\begin{multline*}
\left| P_{\bW_0(x_0)^\perp}(x-x')\right| = \sqrt{\sum_{i=1}^{n-m}\left|\la \bv_i(x_0),x-x' \ra \right|^2} \\ \leq \sqrt{\sum_{i=1}^{n-m}\left|\la \bv_i(x'),x-x' \ra \right|^2}+\sqrt{\sum_{i=1}^{n-m}\left|\la \bv_i(x') - \bv_i(x_0),x-x' \ra \right|^2}\\
\leq \frac{\tau}{2}|x-x'| + \sqrt{n-m}\Lambda |x'-x_0||x-x'| \leq \tau |x-x'|\,.
\end{multline*}
\end{proof}

\begin{Proposition}
\label{upper.bound}
There are $\bdelta_{\theTheorem}(n,\Lambda) > 0$ and $\bc_{\theTheorem}(m) \geq 1$ with the following property. If $u \in U$ and $\rmdiam (E \cup \{u\}) \leq \bdelta_{\theTheorem}(n,\Lambda)$ then
\begin{equation*}
\max \left\{ \calY^0_E \bW(u) , \calY_E \bW(u)\right\} \leq \bc_{\theTheorem}(m) (\rmdiam E)^m \,.
\end{equation*}
\end{Proposition}

\begin{proof}
Let $\bdelta_{\ref{upper.bound}}(n,\Lambda) =\bdelta_{\ref{ub.1}}(n,\Lambda,1/2)$.
Recall the definitions of $\calY^0_E \bW$ and $\calY_E \bW$ from \ref{def.Y0} and \ref{def.Y} respectively. If $E =\emptyset$ the conclusion is obvious. If not pick $x_0 \in E$ arbitrarily. Given any $y \in \R^{n-m}$ we see that \ref{ub.1} applies with $\tau = 1/2$ and in turn the bow-tie lemma \ref{bow.tie} applies to $S = E \cap g_{\bv_1,\ldots,\bv_{n-m},u}^{-1}\{y\}$ and $W = \bW_0(x_0)$. Thus
\begin{equation*}
\calH^m \left(E \cap g_{\bv_1,\ldots,\bv_{n-m},u}^{-1}\{y\} \right) \leq \left( \frac{2}{\sqrt{3}}\right)^m \balpha(m)r^m (\rmdiam E)^m \,.
\end{equation*}
The proposition is proved.
\end{proof}

\begin{Corollary}
\label{cor.ub}
There are $\bdelta_{\theTheorem}(n,\Lambda) > 0$ and $\bc_{\theTheorem}(n) \geq 1$ with the following property. If $\rmdiam E \leq \bdelta_{\theTheorem}(n,\Lambda)$ then
\begin{equation*}
\calZ_E \bW(u) \leq \bc_{\theTheorem}(n) (\rmdiam E)^m
\end{equation*}
for $\calL^n$ almost every $u \in E$.
\end{Corollary}

\begin{proof}
Let $\bdelta_{\ref{cor.ub}}(n,\Lambda) = \min\{ \bdelta_{\ref{upper.bound}}(n,\Lambda),\bdelta_{\ref{Z.1}}(n,\Lambda,1/2) \}$. 
\end{proof}

\section{Lower bound for $\calY_E \bW$ and $\calZ_E \bW$}

\begin{Empty}[Setting for this section]
\label{51}
We enforce again the exact same assumptions as in \ref{31}, and as in \ref{fibration.2} we let $C_r = \R^{n-m} \cap \{ y : |y| \leq r \}$.
\end{Empty}

\begin{Empty}[Polyballs]
\label{pb}
Given $x_0 \in \Rn$ and $r > 0$ we define
\begin{equation*}
\bC_\bW(x_0,r) = \Rn \cap \left\{ x : \left| P_{\bW_0(x_0)}(x-x_0)\right| \leq r \text{ and } \left| P_{\bW_0(x_0)^\perp}(x-x_0)\right| \leq r \right\} \,.
\end{equation*}
We notice that if $x \in \bC_\bW(x_0,r)$ then $|x-x_0| \leq r \sqrt{2}$, in particular $\rmdiam \bC_\bW(x_0,r) \leq 2 \sqrt{2}$. We also notice that $\calL^n\left( \bC_\bW(x_0,r)\right) = \balpha(m)\balpha(n-m)r^n$. 
\end{Empty}

\begin{Empty}
\label{53}
{\it 
Given $0 < \veps < 1$ there exists $\bdelta_{\theTheorem}(n,\Lambda,\veps) > 0$ with the following property. If
\begin{enumerate}
\item $0 < r < \bdelta_{\theTheorem}(n,\Lambda,\veps)$;
\item $u \in \bC_\bW(x_0,r) \subset U$;
\item $|g_{\bv_1,\ldots,\bv_{n-m},u}(x_0)| \leq (1-3\veps)r$;
\item $C \subset C_{\veps r}$ is closed; 
\end{enumerate}
then
\begin{equation*}
\calL^n \left( \bC_\bW(x_0,r) \cap g_{\bv_1,\ldots,\bv_{n-m},u}^{-1}(C)\right) \geq \frac{1}{1+\veps} \balpha(m) r^m \calL^{n-m}(C) \,.
\end{equation*}
}
\end{Empty}

\begin{Remark}
With hopes that the following will help the reader form a geometrical imagery: Under the circumstances \ref{53}, $\bC_\bW(x_0,r) \cap g_{\bv_1,\ldots,\bv_{n-m},u}^{-1}(C)$ may be seen as a <<nonlinear stripe>>, <<horizontal>> with respect to $\bW_0(x_0)$, <<at height>> $g_{\bv_1,\ldots,\bv_{n-m},u}(x_0)$ with respect to $x_0$, and of <<width>> $C$.
\end{Remark}

\begin{proof}[Proof of \ref{53}]
Given $z \in \bW_0(x_0) \cap \bB(0,r)$ we define 
\begin{equation*}
V_z = \Rn \cap \left\{ x_0 + z + \sum_{i=1}^{n-m} y_i \bv_i(x_0) : y \in C_r \right\} \subset \bC_\bW(x_0,r)
\end{equation*}
and we consider the isometric parametrization $\gamma_z : C_r \to V_z$ defined by the formula
\begin{equation*}
\gamma_z(y) = x_0 + z + \sum_{i=1}^{n-m} y_i \bv_i(x_0) \,.
\end{equation*}
We also abbeviate $f_{z,u} = g_{\bv_1,\ldots,\bv_{n-m},u} \circ \gamma_z$.
\par 
\textsc{Claim \#1.} $\rmLip f_{z,u} \leq (1 + \veps)^\frac{1}{n-m}$.
\par 
Since $\gamma_z$ is an isometry it suffices to obtain an upper bound for $\rmLip g_{\bv_1,\ldots,\bv_{n-m},u}|_{\bC_\bW(x_0,r)}$. Let $x,x' \in \bC_\bW(x_0,r)$,
\begin{equation*}
\begin{split}
\big| g_{\bv_1,\ldots,\bv_{n-m},u}(x)& - g_{\bv_1,\ldots,\bv_{n-m},u}(x') \big|  = \sqrt{\sum_{i=1}^{n-m} \left| \la \bv_i(x),x-u\ra - \la \bv_i(x'),x'-u\ra \right|^2} \\
& \leq \sqrt{\sum_{i=1}^{n-m} \left( \left| \la \bv_i(x)-\bv_i(x'),x-u\ra \right| + \left| \la \bv_i(x'),x-x' \ra \right|\right)^2} \\
&\leq \sqrt{\sum_{i=1}^{n-m}\left| \la \bv_i(x)-\bv_i(x'),x-u\ra \right|^2} + \sqrt{\sum_{i=1}^{n-m}\left| \la \bv_i(x'),x-x' \ra \right|^2}\\
& \leq \sqrt{n-m} \Lambda |x-x'||x-u| + \left| P_{\bW_0(x')^\perp}(x-x') \right| \\
& \leq \left( 1 + \sqrt{n-m} \Lambda 2 \sqrt{2}r \right) |x-x'| \,.
\end{split}
\end{equation*}
Recalling hypothesis (1) it is now apparent that $\bdelta_{\ref{53}}$ can be chosen small enough according to $n$, $\Lambda$ and $\veps$ so that \textsc{Claim \#1} holds.
\par 
\textsc{Claim \#2.} {\it For $\calL^{n-m}$ almost every $y \in C_r$ one has $\|Df_{z,u}(y) - \rmid_{\R^{n-m}}\| \leq \veps$}.
\par 
Let $y \in C_r$ be such that $f_{z,u}$ is differentiable at $y$.
We shall estimate the coefficients of the matrix representing $Df_{z,u}(y)$ with respect to the canonical basis. Fix $i,j=1,\ldots,n-m$ and recall \eqref{eq.2}:
\begin{equation*}
\begin{split}
\frac{\partial}{\partial y_i} \la f_{u,z}(y) , e_j \ra & = \frac{\partial}{\partial y_i} \la g_{\bv_1,\ldots,\bv_{n-m},u}(\gamma_z(y)) , e_j \ra \\
& = \frac{\partial g_{\bv_j,u}}{\partial y_i}(\gamma_z(y)) \\
& = \left\la \nabla g_{\bv_j,u}(\gamma_z(y)) , \frac{\partial \gamma_z(y)}{\partial y_i} \right\ra \\
& = \left\la D\bv_j(\gamma_z(y))(\bv_i(x_0)) , \gamma_z(y)-u \right\ra + \left\la \bv_j(\gamma_z(y)) , \bv_i(x_0) \right\ra \\
& = \rmI + \rmII \,.
\end{split}
\end{equation*}
Next notice that
\begin{multline*}
\left| \rmII - \delta_{ij} \right| = \left| \rmII - \la \bv_j(x_0) , \bv_i(x_0) \ra \right| = \left| \la \bv_j(\gamma_z(y)) - \bv_j(x_0) , \bv_i(x_0) \ra \right| \\
\leq \Lambda \left| \gamma_z(y) - x_0 \right| \leq \Lambda 2 \sqrt{2} r \leq \frac{\veps}{2(n-m)}
\end{multline*}
where the last inequality follows from hypothesis (1) upon choosing $\bdelta_{\ref{53}}$ small enough according to $n$, $\Lambda$ and $\veps$. Moreover,
\begin{equation*}
\left| \rmI \right| \leq \Lambda | \gamma_z(y)-u| \leq \Lambda 2 \sqrt{2} r \leq \frac{\veps}{2(n-m)}\,.
\end{equation*}
Therefore if $(a_{ij})_{i,j=1,\ldots,n-m}$ is the matrix representing $Df_{z,u}(y)$ with respect to the canonical basis we have shown that $|a_{ij}-\delta_{ij}| \leq \frac{\veps}{n-m}$ for all $i,j=1,\ldots,n-m$. This completes the proof of \textsc{Claim \#2}.
\par 
\textsc{Claim \#3.} $C_{\veps r } \subset f_{z,u}(C_r)$.
\par 
We shall show that $|y - f_{z,u}(y)| \leq  (1-\veps)r$ for every $y \in \rmBdry C_r$ and the conclusion will become a consequence of the Intermediate Value Theorem in case $m=n-1$, and a standard application of homology theory, see e.g. \cite[4.6.1]{DEP.05c} in case $m < n-1$. If $m < n-1$ it is clearly enough to establish this inequality for $\calH^{n-m-1}$ almost every $y \in \rmBdry C_r$: In that case, owing to the Coarea Theorem \cite[3.2.22]{GMT} we choose such $y$ in order that $f_{z,u}$ is differentiable $\calH^1$ almost everywhere on the line segment $\R^{n-m} \cap \{ sy : 0 \leq s \leq 1 \}$. Whether $m < n-1$ or $m=n-1$ it then follows from \textsc{Claim \#2} that
\begin{multline*}
\left| f_{z,u}(y) - f_{z,u}(0) - y  \right| = \left| \int_0^1 Df_{z,u}(sy)(y) d\calL^1(s) - y \right| \\ \leq \int_0^1 \left| Df_{z,u}(sy)(y)-y\right| d\calL^1(s) \leq \veps |y| = \veps r \,.
\end{multline*}
Accordingly,
\begin{equation*}
\left|f_{z,u}(y) - y \right| \leq \left| f_{z,u}(y) - f_{z,u}(0) - y  \right| + \left| f_{z,u}(0)\right| \leq \veps r + \left| f_{z,u}(0)\right|\,,
\end{equation*}
and the claim will be established upon showing that $\left| f_{z,u}(0)\right| \leq (1- 2 \veps)r$. Note that $f_{z,u}(0) = g_{\bv_1,\ldots,\bv_{n-m},u}(x_0+z)$, and we shall use hypothesis (3) to bound its norm from above. Given $j=1,\ldots,n-m$ recall that $\la \bv_j(x_0) , z \ra = 0$ thus
\begin{multline*}
\left| g_{\bv_j,u}(x_0+z) - g_{\bv_j,u}(x_0) \right| = \left| \la \bv_j(x_0+z) , x_0+z-u \ra - \la \bv_j(x_0) , x_0 - u \ra \right| \\
 = \left| \la \bv_j(x_0+z) , x_0+z-u \ra - \la \bv_j(x_0) , x_0 + z - u \ra \right| \leq \Lambda |z| |x_0 + z - u| \\
 \leq \Lambda r 2\sqrt{2} r \leq \frac{\veps r}{\sqrt{n-m}}
\end{multline*}
where the last inequality holds according to hypothesis (1) provided $\bdelta_{\ref{53}}$ is chosen sufficiently small. In turn 
\begin{multline*}
\left| f_{z,u}(0) \right| \leq \left| g_{\bv_1,\ldots,\bv_{n-m},u}(x_0+z) - g_{\bv_1,\ldots,\bv_{n-m},u}(x_0) \right| + \left| g_{\bv_1,\ldots,\bv_{n-m},u}(x_0)\right| \\
\leq \veps r + (1-3\veps)r = (1-2\veps ) r
\end{multline*}
according to hypothesis (3).
\par 
\textsc{Claim \#4.} {\it For every $z \in \bW_0(x_0) \cap \bB(0,r)$ and every closed $C \subset C_{\veps r}$ one has $\calH^{n-m}(C) \leq (1+\veps) \calH^{n-m} \left( g_{\bv_1,\ldots,\bv_{n-m},u}^{-1}(C) \cap V_z \right)$.}
\par 
First notice that
\begin{equation*}
g_{\bv_1,\ldots,\bv_{n-m},u}^{-1}(C) \cap V_z = \gamma_z \bigg( \gamma_z^{-1} \left(g_{\bv_1,\ldots,\bv_{n-m},u}^{-1}(C) \cap V_z \right) \bigg) = \gamma_z \left( f_{z,u}^{-1}(C)\right)
\end{equation*}
and therefore
\begin{equation*}
\calH^{n-m} \left(g_{\bv_1,\ldots,\bv_{n-m},u}^{-1}(C) \cap V_z \right) = \calH^{n-m}  \left( f_{z,u}^{-1}(C)\right)
\end{equation*}
since $\gamma_z$ is an isometry. Now since $C \subset C_{\veps r} \subset f_{z,u}(C_r)$ according to \textsc{Claim \#3} we have
\begin{equation*}
C = f_{z,u} \left( f_{z,u}^{-1}(C) \right) \,.
\end{equation*}
It therefore follows from \textsc{Claim \#1} that
\begin{equation*}
\begin{split}
\calH^{n-m}(C) & \leq \left( \rmLip f_{z,u} \right)^{n-m} \calH^{n-m}  \left( f_{z,u}^{-1}(C)\right) \\
& \leq (1+\veps) \calH^{n-m} \left(g_{\bv_1,\ldots,\bv_{n-m},u}^{-1}(C) \cap V_z \right)\,.
\end{split}
\end{equation*}
\par 
We are now ready to finish the proof by an application of Fubini's Theorem :
\begin{equation*}
\begin{split}
\calL^n \bigg( \bC_\bW(x_0,r) \cap & g_{\bv_1,\ldots,\bv_{n-m},u}^{-1}(C)\bigg)\\& = \int_{\bW_0(x_0) \cap \bB(0,r)} d\calL^m(z) \calH^{n-m}\left(g_{\bv_1,\ldots,\bv_{n-m},u}^{-1}(C) \cap V_z \right) \\
& \geq \frac{1}{1+\veps} \balpha(m)r^m \calH^{n-m}(C) \,.
\end{split}
\end{equation*}
\end{proof}

\begin{Empty}[Lower bound for $\calY_E \bW$]
\label{54}
{\it 
Given $0 < \veps < 1/3$ there exists $\bdelta_{\theTheorem}(n,\Lambda,\veps)>0$ with the following property. If
\begin{enumerate}
\item $0 < r < \bdelta_{\theTheorem}(n,\Lambda,\veps)$;
\item $\bC_\bW(x_0,r) \subset U$;
\item $A \subset U$ is closed;
\item $\calL^n \left( A \cap \bC_\bW(x_0,r) \right) \geq (1 - \veps) \calL^n \left(\bC_\bW(x_0,r) \right)$;
\end{enumerate}
then
\begin{equation*}
\int_{\bC_\bW(x_0,r)} \calY_{A \cap \bC_\bW(x_0,r)} \bW (u) d\calL^n(u) \geq (1 - \bc_{\theTheorem}(n) \veps) \balpha(m)r^m\calL^n \left( \bC_\bW(x_0,r) \right) \,,
\end{equation*}
where $\bc_{\theTheorem}(n)=5+6n$.
}
\end{Empty}

\begin{proof}
Similarly to the proof of \ref{53} we will first establish a lower bound for $\calY_{A \cap \bC_\bW(x_0,r)} \bW$ on <<vertical slices>> $V_z$ of the given polyball and then apply Fubini. Given $z \in \bW_0(x_0) \cap \bB(0,r)$ we let $V_z$ and $\gamma_z$ be as in \ref{53} and we also define
\begin{equation*}
\check{V}_z = \Rn \cap \left\{ x_0 + z + \sum_{i=1}^{n-m} y_i \bv_i(x_0) : y \in C_{(1-3\veps)r}\right\}
\end{equation*}
(notice it is slightly smaller than $V_z$ used in the proof of \ref{53}) and we consider the isometric parametrization $\check{\gamma}_z : C_{(1-3\veps)r} \to \check{V}_z$ defined by
\begin{equation*}
\check{\gamma}_z(y) = x_0 + z + \sum_{i=1}^{n-m} y_i \bv_i(x_0) \,.
\end{equation*}
For part of the proof we find it convenient to abbreviate $E = A \cap \bC_\bW(x_0,r)$.
We also let $\check {\calY}_{E} \bW = \left( \calY_{E} \bW \right) \circ \check{\gamma}_z$.
\par 
By definition of $\calY_E \bW$ for each $\check{\gamma}_z(y) \in \check{V}_z$ there exists a collection $\calC_y$ of closed balls in $\R^{n-m}$ with the following properties: For every $C \in \calC_y$, $C$ is a ball centered at 0, $C \subset C_{\veps r}$,
\begin{equation*}
\calY_E \bW \left(\check{\gamma}_z(y) \right) + \veps \geq \dashint_C \calH^m \left( E \cap g_{\bv_1,\ldots,\bv_{n-m},\check{\gamma}_z(y)}^{-1}\{h\} \right) d\calL^{n-m}(h) \,,
\end{equation*}
and $\inf \{ \rmdiam C : C \in \calB_y \} = 0$.
Furthermore $\check{\calY}_E \bW$ being $\calL^{n-m}$ summable according to \ref{upper.bound} there exists $N \subset C_{(1-3\veps)r}$ such that $\calL^{n-m}(N)=0$ and every $y \not \in N$ is a Lebesgue point of $\check{\calY}_E \bW$. For such $y$ we may reduce $\calC_y$ if necessary, keeping all the previously stated properties and enforcing that
\begin{equation*}
\dashint_{y+C} \check{\calY}_E \bW d\calL^{n-m} + \veps \geq \left(\check{\calY}_E \bW \right)(y)
\end{equation*}
whenever $C \in \calC_y$.
We infer that for each $y \in C_{(1-3\veps)r} \setminus N$ and each $C \in \calC_y$,
\begin{equation}
\label{eq.30}
\int_{y + C} \check{\calY}_E \bW d\calL^{n-m} + 2\veps \calL^{n-m}(y+C) \geq \int_C \calH^m \left( E \cap g_{\bv_1,\ldots,\bv_{n-m},\check{\gamma}_z(y)}^{-1}\{h\} \right) d\calL^{n-m}(h)\,.
\end{equation}
It follows from the Vitali Covering Theorem that there is a sequence $(y_k)_k$ in $C_{(1-3\veps)r} \setminus N$, and $C_k \in \calC_{y_k}$, such that the balls $y_k + C_k$, $k=1,2,\ldots,$ are pairwise disjoint, and $\calL^{n-m} \left( C_{(1-3\veps)r} \setminus \cup_{k=1}^\infty (y_k+C_k) \right) = 0$. It therefore follows from \eqref{eq.30} and the fact that $\gamma_z$ is an isometry that
\begin{equation}
\label{eq.31}
\int_{V_z} \calY_E \bW d\calH^{n-m} +2\veps \calH^{n-m}(V_z) \geq \sum_{k=1}^\infty \int_{C_k} \calH^m \left( E \cap g_{\bv_1,\ldots,\bv_{n-m},u_k}^{-1}\{y\} \right) d\calL^{n-m}(y)\,,
\end{equation}
where we have abbreviated $u_k = \check{\gamma}_z(y_k)$. We also abbreviate $S_k = g_{\bv_1,\ldots,\bv_{n-m},u_k}^{-1}(C_k)$ and we infer from the coarea formula that for each $k=1,2,\ldots$,
\begin{multline}
\label{eq.32}
\int_{C_k} \calH^m \left( E \cap g_{\bv_1,\ldots,\bv_{n-m},u_k}^{-1}\{y\} \right) d\calL^{n-m}(y) = \int_{E \cap S_k} Jg_{\bv_1,\ldots,\bv_{n-m},u_k}d\calL^n \\
\geq (1-\veps) \calL^n \left( E \cap S_k\right)
\end{multline}
where the last inequality follows from \ref{jac.g} applied with $U = \rmInt \bC_\bW(x_0,r)$ provided that $\bdelta_{\ref{54}}(n,\Lambda,\veps)$ is chosen smaller than $(2\sqrt{2})^{-1}\bdelta_{\ref{jac.g}}(n,\Lambda,\veps)$. Letting $S = \cup_{k=1}^\infty S_k$, and recalling that $E = A \cap \bC_\bW(x_0,r)$, we infer from \eqref{eq.31} and \eqref{eq.32} that
\begin{multline}
\label{eq.33}
\int_{V_z} \calY_E \bW d\calH^{n-m} +2\veps \calH^{n-m}(V_z) \geq (1-\veps) \calL^n(E \cap S) \\
 \geq (1-\veps) \big( \calL^n ( \bC_\bW(x_0,r) \cap S) - \calL^n( \bC_\bW(x_0,r) \setminus A ) \big)\,.
\end{multline}
Applying \ref{53} to each $S_k$ does not immediately yield a lower bound for $\calL^n ( \bC_\bW(x_0,r) \cap S)$ because the $S_k$ are not necessarily pairwise disjoint. This is why we now introduce slightly smaller versions of these:
\begin{equation*}
\check{C}_k = (1-\veps) C_k \quad \text{ and } \quad \check{S}_k = g_{\bv_1,\ldots,\bv_{n-m},u_k}^{-1}\left(\check{C}_k\right) \,.
\end{equation*}
\par 
\textsc{Claim.} {\it The sets $\check{S}_k \cap \bC_\bW(x_0,r)$, $k=1,2,\ldots$, are pairwise disjoint.}
\par 
Assume if possible that there are $j \neq k$ and $x \in \check{S}_j \cap \check{S}_k \cap \bC_\bW(x_0,r)$. Letting $\rho_j$ and $\rho_k$ denote respectively the radius of $C_j$ and $C_k$ we notice that $\rho_j + \rho_k < |y_j-y_k|$ because $(y_j+C_j) \cap (y_k+C_k) = \emptyset$. Since $\check{\gamma}_z$ is an isometry we have $|u_j-u_k| = \left| \check{\gamma}_z(y_j) - \check{\gamma}_z(y_k)\right| = |y_j-y_k|$ and therefore also
\begin{multline}
\label{eq.34}
\left| g_{\bv_1,\ldots,\bv_{n-m},u_j}(x) - g_{\bv_1,\ldots,\bv_{n-m},u_k}(x)\right| \leq \left| g_{\bv_1,\ldots,\bv_{n-m},u_j}(x) \right| + \left| g_{\bv_1,\ldots,\bv_{n-m},u_k}(x) \right|\\
\leq (1-\veps)\rho_j + (1-\veps)\rho_k < (1-\veps) \left| u_j - u_k \right| \,.
\end{multline}
We now introduce the following vectors of $\R^{n-m}$,
\begin{equation*}
h_j = \sum_{i=1}^{n-m} \la \bv_i(x_0),u_j \ra e_i \quad \text{ and } \quad h_k = \sum_{i=1}^{n-m} \la \bv_i(x_0),u_k \ra e_i
\end{equation*}
and we notice that
\begin{equation*}
\left| h_j - h_k \right| = \left| P_{\bW_0(x_0)^\perp}(u_j-u_k)\right| = |u_j-u_k|
\end{equation*}
where the second equality holds because $u_j-u_k \in \bW_0(x_0)^\perp$ as clearly follows from the definition of $\check{\gamma}_z$. Furthermore
\begin{multline*}
\bigg| \big( g_{\bv_1,\ldots,\bv_{n-m},u_j}(x) - g_{\bv_1,\ldots,\bv_{n-m},u_k}(x)\big) - \big( h_j - h_k \big)  \bigg| 
= \sqrt{\sum_{i=1}^{n-m} \left| \la \bv_i(x) - \bv_i(x_0) , u_k-u_j \ra\right|^2}\\
\leq \sqrt{n-m} \Lambda \sqrt{2} r \left| u_j - u_k \right| \leq \veps \left| u_j - u_k \right| \,,
\end{multline*}
since we may choose $\bdelta_{\ref{54}}(n,\Lambda,\veps)$ to be so small that the last inequality holds according to hypothesis (1). Whence
\begin{equation*}
\left| g_{\bv_1,\ldots,\bv_{n-m},u_j}(x) - g_{\bv_1,\ldots,\bv_{n-m},u_k}(x)\right| \geq \left| h_j - h_k \right| -  \veps \left| u_j - u_k \right| = (1-\veps)  \left| u_j - u_k \right|
\end{equation*}
in contradiction with \eqref{eq.34}. The \textsc{Claim} is established.
\par 
Thus
\begin{equation}
\label{eq.35}
\begin{split}
\calL^n \left( \bC_\bW(x_0,r) \cap S\right) & = \calL^n \left( \bC_\bW(x_0,r) \cap \cup_{k=1}^\infty S_k\right) \\
& \geq \calL^n \left( \bC_\bW(x_0,r) \cap \cup_{k=1}^\infty \check{S}_k\right) \\
& = \sum_{k=1}^\infty \calL^n \left( \bC_\bW(x_0,r) \cap  \check{S}_k\right) \\
& = \sum_{k=1}^\infty \calL^n \left( \bC_\bW(x_0,r) \cap g_{\bv_1,\ldots,\bv_{n-m},u_k}^{-1}\left(\check{C}_k\right) \right) \\
& \geq \frac{1}{1+\veps}\balpha(m)r^m \sum_{k=1}^\infty \calL^{n-m} \left( \check{C}_k \right)
\end{split}
\end{equation}
where the last ineqality follows from \ref{53}. We notice that indeed \ref{53} applies since $\check{C}_k \subset C_k \subset C_{\veps r}$ and $\left|g_{\bv_1,\ldots,\bv_{n-m},u_k}(x_0)\right| = \left| P_{\bW_0(x_0)^\perp}(u_k-x_0) \right| = \left| y_k\right| \leq (1-3\veps)r$.
\par 
Now,
\begin{multline}
\label{eq.36}
\sum_{k=1}^\infty \calL^{n-m} \left( \check{C}_k \right)  = (1-\veps)^{n-m} \sum_{k=1}^\infty \calL^{n-m} \left( C_k \right) = (1-\veps)^{n-m} \sum_{k=1}^\infty \calL^{n-m} \left( y_k + C_k \right) \\
\geq (1-\veps)^{n-m} \calL^{n-m} \left( C_{(1-3\veps)r} \right) \geq (1-3\veps)^{2(n-m)} \balpha(n-m)r^{n-m} \,.
\end{multline}
We infer from \eqref{eq.35} and \eqref{eq.36} that
\begin{equation*}
\calL^n \left( \bC_\bW(x_0,r) \cap S\right) \geq \frac{(1-3\veps)^{2(n-m)}}{1+\veps}\calL^n \left( \bC_\bW(x_0,r)\right) \,.
\end{equation*}
It therefore ensues from \eqref{eq.33} and hypothesis (4) that
\begin{equation*}
\int_{V_z} \calY_E \bW d\calH^{n-m} +2\veps \calH^{n-m}(V_z) \geq (1-\veps)^2\frac{(1-3\veps)^{2(n-m)}}{1+\veps}\calL^n \left( \bC_\bW(x_0,r)\right) \,.
\end{equation*}
Integrating over $z$ we infer from Fubini's Theorem
\begin{multline*}
\int_{\bC_\bW(x_0,r)} \calY_{A \cap \bC_\bW(x_0,r)} \bW d\calL^n = \int_{\bW_0(x_0) \cap \bB(0,r)} d\calL^n(z) \int_{V_z} \calY_E \bW d\calH^{n-m} \\
\geq \left( (1-\veps)^2\frac{(1-3\veps)^{2(n-m)}}{1+\veps}-2\veps\right)\balpha(m)r^m\calL^n \left( \bC_\bW(x_0,r)\right) \,.
\end{multline*}
\end{proof}

\begin{Proposition}
\label{lower.bound}
Given $0 < \veps < 1/3$ there exist $\bdelta_{\theTheorem}(n,\Lambda,\veps)>0$ and $\bc_{\theTheorem}(n) \geq 1$ with the following property. If
\begin{enumerate}
\item $0 < r < \bdelta_{\theTheorem}(n,\Lambda,\veps)$;
\item $\bC_\bW(x_0,r) \subset U$;
\item $A \subset U$ is closed;
\item $\calL^n \left( A \cap \bC_\bW(x_0,r) \right) \geq (1 - \veps) \calL^n \left(\bC_\bW(x_0,r) \right)$;
\end{enumerate}
then
\begin{equation*}
\int_{A \cap \bC_\bW(x_0,r)} \calY_{A \cap \bC_\bW(x_0,r)} \bW (u) d\calL^n(u) \geq (1 - \bc_{\theTheorem}(n) \veps) \balpha(m)r^m\calL^n \left( \bC_\bW(x_0,r) \right) \,.
\end{equation*}
\end{Proposition}

\begin{Remark}
The difference with \ref{54} is the domain of integration (being smaller) in the integral on the left hand side in the conclusion.
\end{Remark}

\begin{proof}[Proof of \ref{lower.bound}]
The reader will happily check that 
\begin{equation*}
\bdelta_{\ref{lower.bound}}(n,\Lambda,\veps) = \min \left\{ \bdelta_{\ref{54}}(n,\Lambda,\veps), \left(2\sqrt{2}\right)^{-1} \bdelta_{\ref{upper.bound}}(n,\Lambda)\right\}
\end{equation*}
suits their needs.
\end{proof}

\begin{Proposition}
\label{Z.positive}
There exists $\bdelta_{\theTheorem}(n,\Lambda) > 0$ with the following property. If $\rmdiam E \leq \bdelta_{\theTheorem}(n,\Lambda)$ then
\begin{equation*}
\calZ_E \bW(u) > 0
\end{equation*}
for $\calL^n$ almost every $u \in E$.
\end{Proposition}

\begin{proof}
We let 
\begin{equation*}
\bdelta_{\ref{Z.positive}}(n,\Lambda) = \min \left\{ \bdelta_{\ref{lower.bound}}\left(n,\Lambda,\frac{1}{4\bc_{\ref{lower.bound}}(n)}\right) , \bdelta_{\ref{lb.2}}(n,\Lambda,1/2) \right\} \,.
\end{equation*}
According to \ref{lb.2} it suffices to show that $\calY_E \bW (u) > 0$ for $\calL^n$ almost every $u \in E$. Define $Z = E \cap \{ u : \calY_E \bW(u) = 0 \}$ and assume if possible that $\calL^n(Z) > 0$. Since $Z$ is $\calL^n$ measurable (recall \ref{def.Y}) there exists a compact set $A \subset Z$ such that $\calL^n(A) > 0$. Observe that the sets $\bC_\bW(x,r)$, $x \in U$ and $r > 0$, form a derivation basis for $\calL^n$ measurable subsets of $U$ (because their excentricity is bounded away from zero) thus there exists $x_0 \in A$ and $r_0 > 0$ such that
\begin{equation*}
\calL^n \left( A \cap \bC_\bW(x_0,r) \right) \geq \left(1 - \frac{1}{4\bc_{\ref{lower.bound}}(n)} \right) \calL^n \left( \bC_\bW(x_0,r)\right)
\end{equation*}
whenever $0 < r < r_0$. There is no restriction to assume that $r_0$ is small enough for $\bC_\bW(x_0,r_0) \subset U$. Thus if we let $r= \min \{ r_0 , \bdelta_{\ref{lower.bound}}(n,\Lambda,1/(4 \bc_{\ref{lower.bound}(n)}))\}$ it follows from \ref{lower.bound} that
\begin{equation}
\label{eq.40}
\int_{A \cap \bC_\bW(x_0,r)} \calY_{A \cap \bC_\bW(x_0,r)} \bW (u) d\calL^n(u) \geq \left(1 -  \frac{1}{4}\right) \balpha(m)r^m\calL^n \left( \bC_\bW(x_0,r) \right) > 0 \,.
\end{equation}
On the other hand recalling \ref{def.Y} and the fact that $A \cap \bC_\bW(x_0,r) \subset E$ we infer that $\calY_{A \cap \bC_\bW(x_0,r)} \bW(u) \leq \calY_E(u)$ for all $u \in \Rn$. In particular $\calY_{A \cap \bC_\bW(x_0,r)} \bW(u)=0$ for all $u \in A \cap \bC_\bW(x_0,r) \subset Z$, contradicting \eqref{eq.40}. 
\end{proof}

\section{Proof of the theorem}

\begin{Theorem}
Assume that $S \subset \Rn$, $\bW_0 : S \to \bG(n,m)$ is Lipschitz and $A \subset S$ is Borel. The following are equivalent.
\begin{enumerate}
\item $\calL^n(A)=0$.
\item For $\calL^n$ almost every $x \in A$, $\calH^m(A \cap \bW(x))=0$.
\item For $\calL^n$ almost every $x \in S$, $\calH^m(A \cap \bW(x))=0$.
\end{enumerate}
\end{Theorem}

Recall our convention that $\bW(x) = x + \bW_0(x)$.

\begin{proof}
Since $\bG(n,m)$ is complete we can extend $\bW_0$ to the closure of $S$. Furthermore if the Theorem holds for $\rmClos S$ then it also holds for $S$. Thus there is no restriction to assume that $S$ is closed.
\par 
$(1) \Rightarrow (3)$. It follows from \ref{orth.frame} that each $x \in S$ admits an open neighborhood $U_x$ in $\Rn$ such that $\bW(x)$ can be associated with a Lipschitz orthonormal frame verifying all the conditions of \ref{31} for some $\Lambda_x > 0$. Since $S$ is Lindel\"of there are countably many $x_1,x_2,\ldots$ such that $S \subset \cup_j U_{x_j}$. Letting $E_j = S \cap U_{x_j}$ we infer from \ref{AC.1} that $\phi_{E_j,\bW}$ is absolutely continuous with respect to $\calL^n$. Thus if $\calL^n(A)=0$ then $\calH^m\left( A \cap \bW(x)\right)=0$ for $\calL^n$ almost every $x \in E_j$ by definition of $\phi_{E_j,\bW}$. Since $j$ is arbitrary the proof is complete.
\par
$(3) \Rightarrow (2)$ is trivial.
\par 
$(2) \Rightarrow (1)$ Let $A$ verify condition (3). It is enough to show that $\calL^n(A \cap \bB(0,r)) = 0$ for each $r > 0$. Fix $r > 0$ and define $S_r = S \cap \bB(0,r)$. Consider the $U_{x_j}$ defined in the second paragraph of the present proof; since $S_r$ is compact only finitely many of those, say $U_{x_1},\ldots,U_{x_N}$, cover $S_r$. Let $\Lambda = \max_{j=1,\ldots,N} \Lambda_{x_j}$. Partition each $U_{x_j}$, $j=1,\ldots,N$, into Borel sets $E_{j,k}$, $k=1,\ldots,
K_j$, such that $\rmdiam E_{j,k} \leq \bdelta_{\ref{Z.positive}}(n,\Lambda)$. It then follows from \ref{Z.positive} that
\begin{equation}
\label{eq.50}
\left( \calZ_{A \cap E_{j,k}} \bW \right)(u) > 0
\end{equation}
for $\calL^n$ almost every $u \in A \cap E_{j,k}$. Now fix $j$ and $k$. Observe that $\calH^m \left( A \cap E_{j,k} \cap \bW(x) \right) = 0$ for $\calL^n$ almost every $x \in A \cap E_{j,k}$. Thus $\phi_{A \cap E_{j,k},\bW}(A \cap E_{j,k}) =0$. Moreover,
\begin{equation*}
0 = \phi_{A \cap E_{j,k},\bW}\left(A \cap E_{j,k}\right) = \int_{A \cap E_{j,k}} \left( \calZ_{A \cap E_{j,k}} \bW \right)(u) d\calL^n(u) \,.
\end{equation*}
It follows from \eqref{eq.50} that $\calL^n(A \cap E_{j,k}) =0$. Since $j$ and $k$ are arbitrary, $\calL^n(A)=0$.
\end{proof}

\bibliographystyle{amsplain}
\bibliography{/Users/thierry/Documents/LaTeX/Bibliography/thdp}

%\printindex

\end{document}